\documentclass[letterpaper,12pt]{amsart}
\usepackage{amscd, amssymb}
\usepackage[driverfallback=hypertex]{hyperref}
\usepackage{comment}
\usepackage{changepage}
\usepackage{enumitem}
\usepackage{wrapfig}
\usepackage{tikz}
\usepackage{mathtools}
\usepackage{graphicx}
\usetikzlibrary{decorations.markings}
\setlist[enumerate,1]{label={(\alph*)}}
\input xy
\xyoption{all}

\newtheorem{tm}{Theorem}[section]
\newtheorem{pr}[tm]{Proposition}
\newtheorem{lm}[tm]{Lemma}
\newtheorem{cy}[tm]{Corollary}

\theoremstyle{definition}
\newtheorem{df}[tm]{Definition}

\theoremstyle{remark}
\newtheorem{rem}[tm]{Remark}

\newcommand{\R}{\mathbb R}
\newcommand{\C}{\mathbb C}
\newcommand{\N}{\mathbb N}

\renewcommand{\P}{\mathbb P}
\newcommand{\BL}{\mathbb L}

\renewcommand{\H}{\mathcal H}

\newcommand{\p}{\tilde e}
\renewcommand{\k}{\mathfrak k}

\newcommand{\U}{\mathcal U}

\newcommand{\W}{\mathcal W}

\DeclareMathOperator{\im}{Im}
\DeclareMathOperator{\re}{Re}
\let\re\undefined

\DeclareMathOperator{\re}{Re}
\DeclareMathOperator{\id}{id}
\DeclareMathOperator{\vol}{vol}

\DeclareMathOperator{\Diff}{Diff}

\DeclareMathOperator{\hol}{Hol}

\DeclareMathOperator{\supp}{supp}

\DeclareMathOperator{\ham}{Ham}

\DeclareMathOperator{\hess}{Hess^{N}}
\DeclareMathOperator{\ric}{Ric}

\newcommand{\OO}{\mathcal O}
\newcommand{\HH}{\mathcal H}

\newcommand{\bL}{\mathbb L}
\renewcommand{\i}{\sqrt{-1}}
\newcommand{\T}[1]{\mathcal T_{#1}}

\newcommand{\vp}{\varphi}
\newcommand{\ii}{\sqrt{-1}}
\newcommand{\cP}{\mathcal P}

\newcommand{\ef}{\mathfrak{e}}

\title{Quantization of positive Lagrangians}

\author[J. Solomon]{Jake P. Solomon}
\address{Institute of Mathematics, Hebrew University}
\email{jake@math.huji.ac.il}

\author[B. Yavetz]{Bentzion Yavetz}
\address{Institute of Mathematics, Hebrew University}
\email{bentzion.yavetz@mail.huji.ac.il}
\begin{document}

\date{October 2026}
\begin{abstract}
A Lagrangian submanifold in an almost Calabi-Yau manifold is called positive if the real part of the holomorphic volume form restricted to it is positive. It is called $\epsilon$-positive if the ratio of the restrictions of the imaginary and real parts of the holomorphic volume form is bounded by $\frac{1}{\epsilon}.$ An exact isotopy class of positive Lagrangian submanifolds admits a natural Riemannian metric. We prove that the Riemannian metric restricted to the subset of $\epsilon$-positive Lagrangians induces a path metric when the relative K\"ahler class of the ambient Calabi-Yau manifold lies on a line of rational slope. 

The proof relies on the construction of a sequence of quantization maps from a circle bundle over the space of positive Lagrangians to the dual of the space of holomorphic sections of powers of a prequantum line bundle. The rationality condition for the relative K\"ahler class allows us to reduce to the case of Bohr-Sommerfeld Lagrangians. Building on ideas of Tian and Nirenberg-Wells, H\"ormander estimates are used to show that functions on a Bohr-Sommerfeld Lagrangian can be approximated by holomorphic sections of a sufficiently high power of the prequantum line bundle. Consequently, the quantization maps eventually separate fibers of the circle bundle. 
\end{abstract}

\maketitle

\pagestyle{plain}

\tableofcontents

\section{Introduction}
Let $X$ be a compact almost Calabi-Yau manifold of complex dimension $n.$ Namely, $X$ is a smooth manifold equipped with a complex structure $J,$ a K\"ahler form $\omega,$ and a nowhere vanishing holomorphic $(n,0)$ form $\Omega.$ An oriented Lagrangian submanifold $\Gamma \subset X$ is called \emph{positive} if $\re\Omega|_\Gamma > 0.$ This condition is also known as being almost-calibrated~~\cite{Wan01}.

Let $\OO$ be an exact isotopy class of closed positive Lagrangians. That is, $\OO$ is the space of all closed positive Lagrangian submanifolds $\Gamma \subset X$ that can be moved by a Hamiltonian flow through positive Lagrangians to a fixed $\Gamma_* \subset X.$ Assume $\Gamma_*$ is connected.
Let $\Gamma \in \OO$ and let
\[
\T{\Gamma} = \left\{h \in C^\infty(\Gamma,\R)\left| \int_\Gamma h \re \Omega = 0\right.\right\}.
\]
There is a canonical isomorphism $T_\Gamma \OO \cong \T{\Gamma}.$ Thus we define a Riemannian metric $G$ on $\OO$ by
\[
G(h,k) = \int_\Gamma hk \re \Omega, \qquad h,k \in \T{\Gamma}.
\]
It was shown in~~\cite{So13a} that $G$ has non-positive sectional curvature. Initial results on existence of geodesics for the metric $G$ have been obtained in~~\cite{RuS15,SoY15,SoY20a,SoY20b}. However it remains an open question to determine whether $G$ induces a non-degenerate path metric on $\OO.$ In the present paper we make a first step in this direction.

For $\epsilon > 0,$ a positive Lagrangian $\Gamma \subset X$ is called $\epsilon$-positive if
\[
\epsilon \left | \frac{\im \Omega|_\Gamma}{\re\Omega|_\Gamma}\right| < 1.
\]
A similar condition appears in the definition of almost calibrated in~~\cite{Ne07}. Observe that for any compact positive Lagrangian $\Gamma,$ there exists $\epsilon > 0$ such that $\Gamma$ is $\epsilon$-positive. Suppose $\Gamma_* \in \OO$ is $\epsilon$-positive. Let $\OO^\epsilon \subset \OO$ denote the open subspace of $\epsilon$-positive Lagrangians that can be moved by a Hamiltonian flow through $\epsilon$-positive Lagrangians to $\Gamma_*.$ So, $\OO^\epsilon$ inherits the Riemannian metric $G$ from $\OO.$
Our main result is the following.
\begin{tm}\label{tm:pm}
Assume that $[\omega] \in H^2(X,\Gamma_*;\R)$ lies on a line of rational slope. Then the Riemannian metric $G$ induces a non-degenerate path metric on $\OO^\epsilon.$
\end{tm}

The non-degeneracy of the path metric induced by the restriction of $G$ to Hamiltonian isotopy classes of $O_n$-invariant positive Lagrangian spheres in $n$-dimensional $A_m$ Milnor fibers was proved in~\cite{SoY15} without any assumption on $\omega$.

If $[\omega] \in H^2(X,\Gamma_*;\R)$ lies on a line of rational slope, then for any $\Gamma \in \OO$ also $[\omega] \in H^2(X,\Gamma;\R)$ lies on a line of rational slope. So, the assumption of Theorem~\ref{tm:pm} depends only on $\OO.$ After multiplying by a suitable positive constant we may assume without loss of generality that $[\omega]\in H^2(X,\Gamma;\R)$ is integral for any $\Gamma\in \OO$.

The main ingredient in the proof of Theorem~~\ref{tm:pm} is the following family of quantization maps. By~\cite[Lemma~6]{Ve20}, choose a holomorphic line bundle $\bL \to X$ with a Hermitian metric $h$ and Chern connection $D$ such that the curvature satisfies $F_D = -2\pi\i  \omega$ and the flat bundle $(\bL,D)|_{\Gamma_*}$ has trivial holonomy. That is, $\Gamma_*$ satisfies the Bohr-Sommerfeld condition. It follows that $(\bL,D)|_{\Gamma}$ has trivial holonomy for all $\Gamma \in \OO.$
Definition~\ref{df:cP} gives a smooth $U(1)$ bundle
\[
\pi_\cP : \cP \to \OO^\epsilon
\]
which satisfies
\begin{equation}
\begin{gathered}
\cP = \{(\Gamma,e) \,|\, \Gamma \in \OO^\epsilon, \, e \in C^\infty(\bL|_{\Gamma}), \, De = 0, \, |e|_h = 1\},  \label{eq:cP}\\
    \pi_\cP(\Gamma,e) = \Gamma.
\end{gathered}
\end{equation}
Definition~\ref{df:metric} gives a Riemannian metric $\overline G$ on $\cP$, for which the projection $\pi:\cP\to \OO$ is a Riemannian submersion. Let
\[
\H_m = H^0(X,\bL^{\otimes m}).
\]
For $(\Gamma,e) \in \cP,$
define a linear functional
\[
\psi_{\Gamma,e}: \H_m \to \C
\]
by
\[
\psi_{\Gamma,e}(s) = \int_\Gamma \frac{s}{e^m} \Omega.
\]
Define a map
\[
\vp_m : \cP \to \H_m^\vee
\]
by
\[
\vp_m(\Gamma,e) = \psi_{\Gamma,e} \in \H_m^\vee.
\]
Theorem~~\ref{tm:pm} is a consequence of the following two results on $\vp_m.$
\begin{tm}\label{tm:bound}
There exists a constant $C_m$ depending only on $m$ and the geometry of $X$ such that
\[
\| d\vp_m\| \leq C_m.
\]
\end{tm}

\begin{tm}\label{tm:separation}
For any distinct Lagrangians $\Gamma_0,\Gamma_1 \in \OO^\epsilon$ there exists $M \in \N$ such that $\vp_m(\Gamma_0,e_0) \neq \vp_m(\Gamma_1,e_1)$ for all $m \geq M$ and all $e_0\in \cP_{\Gamma_0},e_1\in\cP_{\Gamma_1}$.
\end{tm}

In the case where $\Gamma_0,\Gamma_1$ have a clean and connected intersection, Theorem~\ref{tm:separation} can be deduced from \cite[Theorem 1.2]{Io22}.
We prove Theorem~\ref{tm:separation} using the following holomorphic approximation theorem.
\begin{tm}\label{tm:approximation1}
    For any $(\Gamma,e)\in \cP$, any $b\in C^\infty(\Gamma,\C),r\in \N$ and $\varepsilon>0$ and any open $V\subseteq X$ such that $\supp b\subseteq V\cap \Gamma,$ there exists $M\in\N$ such that for all $m\ge M$ there exists $s\in \H_m$ with
    \[
    \left\|s|_\Gamma-be^m\right\|_{C^r}<\varepsilon, \qquad\left\|s|_{X\setminus V}\right\|_{C^r}<\varepsilon.
    \]
\end{tm}
Theorem~\ref{tm:approximation1} is a special case of Theorem~\ref{tm:approximation2}.
Versions of this theorem appear in the previous works~\cite[Proposition~2.4]{Char03}, \cite[Proposition~19]{Ve20}, and~\cite[Proposition~2]{Ti25}.

\subsection{Context}
There is also the Hofer metric on the exact isotopy class $\OO$, defined by taking the infimum of the Hofer lengths of Hamiltonian flows taking one Lagrangian to another. Recall that Hofer~\cite{Ho90} introduced a bi-invariant metric on the group of Hamiltonian diffeomorphisms $\ham(X,\omega)$ by defining the length of the Hamiltonian flow $\varphi_t$ associated with a time-dependent Hamiltonian $\{H_t\}_{t \in [0,1]}$ to be
\[
\ell(\varphi_t) =   \int _0^1\left(\sup_X H_t-\inf_X H_t\right)dt.
\]
This is the Finsler metric on $\ham(X,\omega)$ associated to the norm on $C^\infty(X)/\R\cong T_{\id}\ham(X,\omega)$ induced by (twice) the $L^\infty$ norm on $C^\infty(X)$. The Hofer metric on $\ham(X,\omega)$ was shown to be non-degenerate for arbitrary $X$ by Lalonde-McDuff~\cite{LaMc90}. The Hofer metric on $\OO$ is invariant under the action of $\ham(X,\omega)$ on $\OO$, and was shown to be nondegenerate by Chekanov~\cite{Chek00}. There are also $L^p$-analogs of these norms, shown to be degenerate on $\ham(X,\omega)$ by Eliashberg and Polterovich~\cite{ElPo93}, and on $\OO$ by Chekanov~\cite{Chek00}.

The metric considered here arises from an $L^2$ norm, so one might expect it to be degenerate by analogy to the $L^p$ analogs of the Hofer norm. Yet it is non-degenerate, at least when restricted to $\OO^\epsilon$. For this to be possible, it is important that the metric considered here is not $\ham(X,\omega)$-invariant.
Rather, it is analogous to the Mabuchi metric~\cite{Mab87} on the space of K\"ahler metrics in a fixed cohomology class of a compact manifold, which also arises from an $L^2$ norm. The Mabuchi metric was proven to be non-degenerate by Chen~\cite{Che00}. 

Geometric quantization allows the approximation infinite-dimensional objects by finite-dimensional ones in a number of settings. Tian~\cite{tian1990polarized} used a prequantization line bundle over a compact K\"ahler manifold $L \to M$ to show that any K\"ahler metric in $c_1(L)$ can be obtained as the limit of Bergman metrics, that is, rescaled pullbacks of Fubini-Study metrics by Kodaira embeddings $M\to \P(H^0(M,L^m)^\vee)$ for large $m$. Donaldson~\cite{Don01} showed that constant scalar-curvature K\"ahler metrics arise as the limits of balanced metrics, which are Bergman metrics satisfying an explicit finite dimensional condition. Arezzo-Tian~\cite{ArT03} and Phong-Sturm~\cite{PhS06} studied the geodesics of the infinite-dimensional space of K\"ahler potentials using the geodesics of the finite-dimensional spaces of Bergman metrics.

An analogous quantization program exists for the infinite-dimensional manifold of Hermitian metrics on a holomorphic vector bundle $E \to Y$ of rank $r$. Given a prequantization line bundle $L$, for large $m$ there is an embedding to the Grassmannian $f: Y\to G\left(r,H^0(E\otimes L^m)^\vee\right)$ such that $E = f^* U^\vee \otimes L^{-m}$ where $U\to G\left(r,H^0(E\otimes L^m)^\vee\right)$ is the universal bundle. Thus, one obtains a family of pullback metrics on $E.$ Again, every metric on $E$ is the limit of such rescaled pullback metrics. See Hashimoto-Keller~\cite{HaKe25}, based on Catlin~\cite{Cat99} and Wang~\cite{WaX02}. See also Catlin-D'Angelo~\cite{CatD99}, Wang~\cite{WaX05} and Thomas~\cite{Th06}.

There is a substantial literature on Lagrangian submanifolds and quantization. Borthwick–Paul–Uribe~\cite{BoPaUr95} studied states in $\H_m$ associated to Bohr-Sommerfeld Lagrangians equipped with half-forms. Ioos~\cite{Io22} studied states associated to Bohr-Sommerfeld Lagrangians equipped with non-vanishing top-forms. If one takes $\re\Omega|_\Gamma$ as the non-vanishing top-form, Ioos' state is the Riesz representatives of the quantization functional $\psi_{\Gamma,e}$.  Paoletti~\cite{Pao06} built an object similar to $\cP$ using half-weights.

Within the framework of homological mirror symmetry~\cite{Ko95}, there is a rough analogy between Lagrangian submanifolds of $X$ and holomorphic vector bundles on a mirror Calabi-Yau manifold $Y$. Thus, it would be interesting to understand to what extent the quantization construction of the present paper corresponds under mirror symmetry to the quantization of holomorphic vector bundles. To the extent such a correspondence holds, it appears that mirror symmetry reverses the direction of the quantization map. Indeed, the quantization of a holomorphic vector bundle $E$ gives a sequence of maps from the finite-dimensional space of linear metrics on $U\to G\left(r,H^0(E\otimes L^m)^\vee\right)$ to the infinite-dimensional space of Hermitian metrics on $E$. On the other hand, the quantization of positive Lagrangians gives a sequence of maps from the infinite-dimensional space $\cP$ to the finite-dimensional spaces~$\H_m^\vee$.

\subsection{Outline}
In Section 2 we show that $\cP$ is a smooth circle bundle over $\OO$ and identify its tangent space at $(\Gamma,e)$ with the smooth functions on $\Gamma$. Furthermore, we construct a flat connection on $\cP.$ Proposition~\ref{pr:der} gives an explicit formula for the derivative of the quantization map $\varphi_m$. This is used to prove the bound of Theorem~\ref{tm:bound}. 

In Section 3 we prove a generalized form of the approximation Theorem~\ref{tm:approximation1}. The proof mimics the proof of Lemma~1.2 of Tian~\cite{tian1990polarized} - use of H\"ormander $L^2$ estimates to construct peak sections at a point $x_0$. An important step in Tian's construction is a local holomorphic section $e$ around $x_0$, with the norm $|e|$ having a strong local maximum at $x_0$. This implies that the norm of $e^m$ decays fast in annuli around $x_0$ as $m\to\infty$, which in turn implies certain $L^2$-bounds. In our case such a holomorphic section does not exist, and we replace it by an almost-holomorphic extension of $e_0\in \cP_{\Gamma_0}$ with the norm having a strong local maximum at $\Gamma_0$ on each fiber of an appropriately chosen tubular neighborhood. Lemma~\ref{lm:BSextpara} asserts the existence of such an extension, using Lemma~3.1 of Nirenberg-Wells~\cite{NiW69}. The proof concludes with the use of elliptic regularity and Sobolev methods as in~\cite[Theorem~6.1]{NiW69}.

In Section 4 we prove the separation of fibers of Theorem~\ref{tm:separation} by a somewhat straight-forward application of Theorem~\ref{tm:approximation1}, and then the non-degeneracy of the main Theorem~\ref{tm:pm}. For the latter we use flatness of $\cP$ to lift any path from $\OO$ to $\cP$ without changing its length, and then pushing forward through the quantization map - which changes the length only up to a constant, by Theorem~\ref{tm:bound} - and relying on Theorem~\ref{tm:separation}.

\subsection{Acknowledgments}
The authors are grateful to S.~Donaldson, Y.~Rubinstein, G.~Tian and S. Zelditch for helpful conversations. ChatGPT Astra was used to proofread the manuscript, correct a number of errors and find additional references.
During the preparation of this manuscript, the authors were partially supported by BSF grant~2024293. The first author was partially supported by BSF Grants 2012236 and 2020329 and the Miriam and Julius Vinik Chair in Mathematics.

\section{Derivative of quantization map}\label{sec:der}
In this section we prove Theorem~\ref{tm:bound}. We start with a discussion of the smooth structure of $\cP$ and its tangent spaces. This allows the definition of a metric and a flat connection on $\cP$, which we will use in the proof of the main Theorem~\ref{tm:pm} in Section~\ref{sec:Pfs}. We further use the results on $\cP$ and $T\cP$ to give an explicit formula for $d\varphi_m$, and conclude with bounding it, hence proving Theorem~\ref{tm:bound}.

\subsection{The bundle \texorpdfstring{$\cP$}{P}}
\subsubsection{Background on \texorpdfstring{$\OO$}{O}}
We start by briefly explaining the smooth structure of $\OO$. See~\cite[Theorem~2.12]{SoY20a} for more details. It is comfortable to have an $n$-dimensional manifold $L$ to which all the Lagrangians in $\OO$ are diffeomorphic. For any $\Gamma\in \OO$ there is a Weinstein neighborhood $W$ in $X$. That means there is a neighborhood $V$ of the zero section in $T^*\Gamma$ and a symplectomorphism $\eta:V\to W$. We identify $W=V\subseteq T^*X$ using $\eta$, and we may assume that each fiber $W_x=T^*_x\Gamma\cap W$ is a ball around $0$. Then $W$ induces a neighborhood $\Gamma\in \W\subseteq \OO$ in the $C^\infty$ topology, in which every Lagrangian is the image in $V $ of an exact 1-form on $\Gamma$. This gives coordinates for $\W$ modeled on the Fr\'echet space
\[
\T{\Gamma}=\left\{h\in C^\infty(\Gamma):\int h\re\Omega=0\right\},
\]
by identifying the image of an exact 1-form $\alpha$ with the unique $k\in\T{\Gamma}$ for which $dk=\alpha$, and we write $\Gamma_k:=\im dk$.
Thus, for a path $\Gamma_t=\Gamma_{k_t}=\im dk_t,$ we have $\dot\Gamma_0=\dot k_0\in \T{\Gamma}$.

The derivative of a path in $\OO$ can be expressed more invariantly as explained in~\cite{AS01}. Let $i_t:L\to I\times L$ be the map $i_t(x)=(t,x)$. It can then be shown that a path $\left(\Gamma_t\right)_{t\in I}$ in $  \OO$ is smooth if and only if there is a smooth map $f:I\times L\to X$ such that for each $t$, $f_t=f\circ i_t:L\to X$ is a representative of $\Gamma_t$, that is, an embedding with $\im f_t=\Gamma_t$. Any other representative will be of the form $f_t\circ\varphi_t$ for $\varphi_t\in\Diff(L) $.

\begin{lm}[{\cite[Lemma~2.1]{AS01}}]\label{lm:Otanspce}
    There is a canonical identification 
    \[
    T_{\Gamma}\OO \cong \T{\Gamma}
    \]
    defined as follows: given a path $\left(\Gamma_t\right)_{t}$ in $\OO$ with a representative $f:I\times L\to X$, we have a vector field $\dot f_t$ along $f_t$, which can be pushed forward to a vector field $v_t=(f_t)_*\dot f_t$ along the inclusion $\Gamma_t\subset X$. Define $\alpha_t=\omega(v_t,\cdot)\in \Omega^1(\Gamma_t) $. Then $\alpha_t$ is exact, and we identify $\dot\Gamma_t$ with the unique primitive $h_t\in\T{\Gamma_t}$.
\end{lm}

\subsubsection{The smooth structure of \texorpdfstring{$\cP$}{P}}
To define a smooth structure on $\cP$ we view it as a bundle associated to a unitary local system obtained from $\BL$ by pullback to the universal Lagrangian bundle over $\OO,$ twisting the connection by a $1$-form, and pushing forward to $\OO.$

Define the universal $L$-bundle $\pi_\U:\U \to \OO$ by
\[
\U=\{(\Gamma,x)\in \OO\times X:x\in \Gamma\},\qquad \pi_\U(\Gamma,x)=\Gamma.
\]
A local trivialization for $\U$ over an open set $\W \subset \OO$ coming from a Weinstein neighborhood $W$ of $\Gamma,$ is given by the map $\W\times \Gamma\to \U|_\W$ defined by $(\Gamma_k,x)\mapsto (\Gamma_k,dk_x)$.
The proof of~\cite[Theorem 2.12]{SoY20a} can be adapted to show compatibility of these trivializations, so they induce a smooth structure on $\U$. Let $p:\U\to X$ be the map $p(\Gamma,x)=x$. 

Let $\chi$ be the canonical $1$-form on $\U$ given as follows. For $(\Gamma,x) \in \U$ and $v \in T_{\Gamma,x}\U$ we identify $d\pi_\U(v)$ with a function on $\Gamma$ by Lemma~\ref{lm:Otanspce}. Then
\begin{equation}\label{eq:chi}
\chi_{\Gamma,x}(v) = 2 \pi \ii d\pi_\U(v)(x).
\end{equation}
The pullback bundle $p^*\BL$ comes with a pullback connection $p^*D,$ but this connection is not flat. To remedy the situation, we define the twisted connection $D_\chi=p^*D-\chi.$

Let $W$ be a Weinstein neighborhood of $\Gamma$ as above. For each $\Gamma_k\in\W$ we have a diffeomorphism $f_k:\Gamma\to \Gamma_k$ by $f_k(x)=dk_x$. since $\W$ is a coordinate neighborhood with $\T{\Gamma}$ as a model space, there is an isomorphism $T_{\Gamma_k}\W\cong \T{\Gamma} $. Composing with the isomorphism $T_{\Gamma_k}\OO\cong \T{\Gamma_k}$ of Lemma~\ref{lm:Otanspce}, we get an isomorphism $\T{\Gamma}\to \T{\Gamma_k}$, which is given by
\[
h\mapsto h-\int_\Gamma hf_k^*\re\Omega.
\]
This motivates the definition of a 1-form $\zeta$ on $\W$ by setting
\begin{equation}\label{eq:zeta}
    \zeta_{\Gamma_k}(h)=\int_\Gamma hf_k^*\re\Omega,\qquad h\in \T{\Gamma}\cong T_{\Gamma_k}\W.
\end{equation}
So, for a path $\Gamma_t=\Gamma_{k_t}$ we have
\begin{equation}\label{eq:TGammat}
    \dot\Gamma_t=\dot k_t-\zeta_{\Gamma_t}\left(\dot k_t\right)\in \T{\Gamma_t}.
\end{equation}
\begin{lm}\label{lm:zeta}
    The 1-form $\zeta$ is closed. Hence, shrinking $\W$ if necessary, there is a function $F:\W\to \R$ such that $dF=-\zeta$.
\end{lm}
\begin{proof}
    Let $h,l\in \T{\Gamma}$, thought of as constant vector fields on $\W$, so $[h,l]=0$. Then
    \[
    d\zeta(h,l)=h\zeta(l)-l\zeta(h).
    \]
Let $\Gamma_k\in \W$. Then we have
    \[
    h\zeta(l)(\Gamma_k)=\frac{d}{dt}\Big|_{t=0}\zeta_{\Gamma_{k+th}}(l)=\int_\Gamma l\frac{d}{dt}\Big|_{t=0} f_{k+th}^*\re\Omega.
    \]
    Define $v_k=\frac{d}{dt}\big|_{t=0} f_{k+th},$ a vector field along $f_k$. So, $i_{v_h}\omega=dh$. By the Cartan formula,
    \[
    \frac{d}{dt}\Big|_{t=0} f_{k+th}^*\re\Omega=i_{v_h}d\re\Omega+di_{v_h}\re\Omega=di_{v_h}\re\Omega.
    \]
    Hence, we have
    \begin{align*}
        d\zeta(h,l)(\Gamma_k)&=\int_\Gamma\left(ldi_{v_h}\re\Omega-hdi_{v_l}\re\Omega\right)\\
        &=\int_\Gamma \left(dh\wedge i_{v_l}\re\Omega-dl\wedge i_{v_h}\re\Omega\right).
    \end{align*}
    
    To see why this expression vanishes, start with $\omega\wedge\re\Omega=0,$ which holds because $\Omega$ is an $(n,0)$-form and $\omega$ is of type $(1,1)$. Since $v_h,v_l$ are tangent to the fibers of $T^*\Gamma$, which are Lagrangian, we have $i_{v_l}i_{v_h}\omega=0$. Hence,
    \begin{align*}
        0&=i_{v_l}i_{v_h}(\omega\wedge\re\Omega) \\
        &=dh\wedge i_{v_l}\re\Omega-dl\wedge i_{v_h}\re\Omega+\omega|_\Gamma \wedge i_{v_l}i_{v_h}\re\Omega\\
        &=
    dh\wedge i_{v_l}\re\Omega-dl\wedge i_{v_h}\re\Omega.
    \end{align*}
\end{proof}

The restriction of $(\mathbb L,D)$ to each fiber $W_x$ has trivial holonomy since the fiber is Lagrangian and simply-connected. Hence, given a unitary section $e$ of $\BL|_\Gamma$, we can extend it to a unique section $\p$ of $\BL|_W$ such that $\p^x:=\p|_{W_x} $ is parallel. This allows the following construction, where $\pi:T^*\Gamma\to \Gamma$ is the projection.

\begin{df}\label{def:section}
    Given a unitary and parallel section $e$ of $\BL|_\Gamma$, extend it to a unitary section $\overline e$ of $p^*\BL$ over $\U|_\W$ as follows: Let $F$ be as in Lemma~\ref{lm:zeta}. For any $\Gamma_k=\im dk\in \W $ let $K=k\circ \pi\in C^\infty(T^*\Gamma)$, and define
    \[
    \overline e|_{\U_{\Gamma_k}}=\exp\left({2\pi \i (K+F\circ \pi_\U)}\right)\p|_{\Gamma_k}\in C^\infty(\Gamma_k,\BL).
    \]
\end{df}

\begin{lm}\label{lm:parallel}
    The section $\overline e$ is parallel with respect to $D_\chi$.
\end{lm}
\begin{proof}
    Denote $\overline e_{\Gamma_k}=\overline e|_{\U_{\Gamma_k}}$. We will first prove that each $\overline e_{\Gamma_k}$ is parallel. Note that $\chi|_{\U_{\Gamma_k}}=0$, so we work with the original connection $D$. Denote also $k_t=tk$ for each $t\in I=[0,1]$. 

    We need to show that for any path $\gamma:I\to \Gamma_k$ we have $\frac{D}{ds}\overline e_{\Gamma_k}(\gamma(s))=0 $, or equivalently $\overline e_{\Gamma_k}(v)=P^\gamma_{0,1}\overline e_{\Gamma_k}(u) $, where $P^\gamma_{0,1}$ is the parallel transport along $\gamma$ from $u=\gamma(0)$ to $v=\gamma(1)$. Denote $x=\pi(u),y=\pi(v)$, so that $u=dk_x,v=dk_y$.   

    Define the path $\delta=\pi\circ\gamma:I\to \Gamma$ so that $\gamma=dk\circ \delta$. Let $\beta(s)=d(k_{s})_y,\alpha(s)=d(k_{s})_x $ be paths from $y$ to $v$ and from $x$ to $u$ as shown in Figure~\ref{fig:loop}. Let also $\overline \delta(s)=\delta(1-s)$ and the same for $\beta$. Define the loop $\ell =\gamma*\alpha*\overline\delta*\overline\beta$. Note that $\ell$ encloses the surface $\Sigma:I^2\to V$ given by $\Sigma(t,s)=(dk_t)_{\delta(s)} $.

    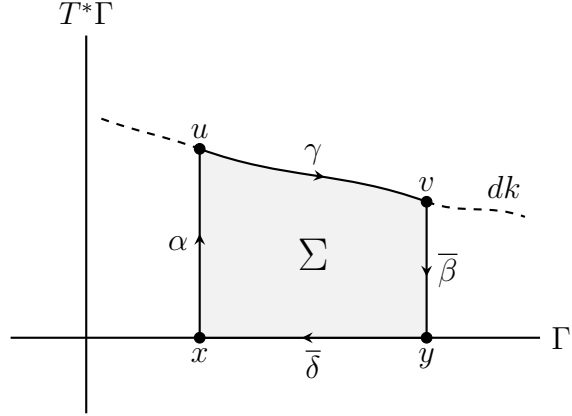
\begin{figure}[bhtp]
  \centering
  \begin{tikzpicture}[
      >=stealth,
      decoration={
          markings,
          mark=at position 0.55 with {\arrow{>}}
      }
  ]
  
  \draw[thick] (-1, 0) -- (6, 0) node[right] {$\Gamma$};
  \draw[thick] (0, -1) -- (0, 4) node[above] {$T^*\Gamma$};

  \coordinate (x) at (1.5, 0);
  \coordinate (y) at (4.5, 0);
  \coordinate (u) at (1.5, 2.5);
  \coordinate (v) at (4.5, 1.8);

  \fill[gray!10] (x) -- (u) to[out=-20, in=160] (v) -- (y) -- cycle;
  \node at (3, 1.1) {\Large $\Sigma$};

  \draw[thick, postaction={decorate}] (x) -- (u) node[midway, left] {$\alpha$};
  \draw[thick, postaction={decorate}] (u) to[out=-20, in=160] node[midway, above] {$\gamma$} (v);
  \draw[thick, postaction={decorate}] (v) -- (y) node[midway, right] {$\overline\beta$};
  \draw[thick, postaction={decorate}] (y) -- (x) node[midway, below] {$\overline{\delta}$};

  \draw[thick, dashed] (0.2, 2.9) to[out=-20, in=160] (u);
  \draw[thick, dashed] (v) to[out=-20, in=160] node[above right] {$dk$} (5.8, 1.6);

  \filldraw (x) circle (2pt) node[below] {$x$};
  \filldraw (y) circle (2pt) node[below] {$y$};
  \filldraw (u) circle (2pt) node[above] {$u$};
  \filldraw (v) circle (2pt) node[above] {$v$};

  \end{tikzpicture}
  \caption{The loop $\ell$ bounding the surface $\Sigma$.}
  \label{fig:loop}
\end{figure}

    Let $\lambda$ be the canonical 1-form on $T^*\Gamma$. Using the facts that $\p|_\Gamma=e$ is parallel and $\p^x,\p^y $ are parallel extensions of $e(x),e(y)$ together with the Ambrose-Singer formula, we get the following:
    \begin{align*}
    \left(\p(v)\right)^{-1}P^\gamma_{0,1}\p(u)&=\left(\p(v)\right)^{-1}P^\ell_{0,1}\p(v)=\hol_\ell\\
    &=\exp\left(2\pi\i\int_\Sigma \omega\right)=\exp\left(2\pi\i\int_\ell\lambda\right).
    \end{align*}
    Recall that the pullbacks of $\lambda$ to the zero section and the fibers of $T^*\Gamma$ vanish. Moreover, $\mu^*\lambda=\mu$ for each 1-form $\mu$ on $\Gamma$. Hence,
    \begin{align*}
        \int_\ell\lambda&=\int_\gamma\lambda=\int _\delta (dk)^*\lambda=\int _\delta dk\\
        &=k(\delta(1))-k(\delta(0))=k(y)-k(x),
    \end{align*}
    so
    \[
    P^\gamma_{0,1}\p(u)=\exp\left(2\pi\i\left(k(y)-k(x)\right)\right)\p(v),
    \]
    and finally
    \begin{align*}
        P^\gamma_{0,1}\overline e_{\Gamma_k}(u)
        &=P^\gamma_{0,1}\left(\exp\left(2\pi\i F(\Gamma_k)+2\pi\i k(x)\right)\p(u)\right)\\
        &=\exp\left(2\pi\i F(\Gamma_k)+2\pi\i k(y)\right)\p(v)=\overline e_{\Gamma_k}(v).
    \end{align*}
    
It remains to show that $\overline e$ is parallel along paths transversal to the fibers of $\U \to \OO.$ Recall that we have a trivialization $\W\times\Gamma\cong \U|_\W$ by $(\Gamma_{k},x)\mapsto (\Gamma_k,dk_x)$. Consider a path $\Gamma_{k_t} \in \W$ with $k_0 = k.$  The covariant derivative of $\overline e$ in the direction of the path $\gamma(t) = (\Gamma_{k_t},d(k_t)_x)$ transversal to $\U_{\Gamma_k}$ is given by
    \begin{align*}
        \frac{D_\chi}{dt}&\Big|_{t=0}\overline e(\Gamma_{k_t},d(k_t)_x)\\
        &=\left(\frac{D}{dt}-\chi\right)\Big|_{t=0}\exp\left(2\pi\i (k_t(x)+F(\Gamma_{k_t}))\right)\p(d(k_t)_x)\\
        &=\left(2\pi\i\left(\dot k_0(x)+dF\left(\dot k_0\right)\right) -\chi_{\Gamma_{k_0},d(k_0)_x}(\dot \gamma(0))\right)\overline e(\Gamma_{k_t},d(k_t)_x)\\
        &=\left(2\pi\i\left(\dot k_0-\zeta(\dot k_0)\right)-\chi_{\Gamma_{k_0},d(k_0)_x}(\dot \gamma(0))\right)\overline e(\Gamma_{k_t},d(k_t)_x)\\
        &\overset{\mathclap{\eqref{eq:TGammat}}}{=}\left(2\pi\i \dot\Gamma_{k_0}(d(k_0)_x)-\chi_{\Gamma_{k_0},d(k_0)_x}(\dot \gamma(0))\right)\overline e(\Gamma_{k_t},d(k_t)_x),
    \end{align*}
    which is zero by the definition of $\chi$.
\end{proof}

The existence of unitary and parallel local sections establishes the following:

\begin{cy}\label{cy:flat}
    The connection $D_\chi$ is flat and preserves the metric $p^*h$. As a consequence, the unit circle subbundle $S_{p^*\BL}\subset p^*\BL$ is flat.
\end{cy}

In light of Corollary~\ref{cy:flat}, we think of $S_{p^*\BL}$ as a $U(1)$-local system and make the following definition.
\begin{df}\label{df:cP}
    The $U(1)$-bundle $\pi_\cP:\cP\to \OO$ is the flat bundle associated with the local system $(\pi_\U)_*S_{p^*\BL}$.
\end{df}
\begin{rem}\label{rm:cP}
    The fiber of $\cP$ over $\Gamma\in \OO$ consists of the parallel sections of the restriction $S_{p^*\BL}|_{\U_\Gamma} $, that is, unitary and parallel sections of $\BL|_\Gamma$. So, $\cP$ satisfies~\eqref{eq:cP} as claimed.
\end{rem}

\subsubsection{The tangent bundle of \texorpdfstring{$\cP$}{P}}

Lemma~\ref{lm:parallel} and Remark~\ref{rm:cP} allow us to view $\Gamma_k\mapsto \overline e|_{\U_{\Gamma_k}}$ as a flat section of $\cP$ over $\W$, and by abuse of notation we will denote this section simply by $\overline e$.
\begin{lm}\label{lm:tanspce}
   The tangent space to $\cP$ at $(\Gamma,e)$ is canonically identified with $C^\infty(\Gamma)$ as follows: Given a path $\gamma=(\Gamma_t,e_t)_{t\in I},$ a map $f:I\times L\to X$ such that $f_t=f\circ i_t:L\to X$ represents $\Gamma_t$, and $\ef\in C^\infty(I\times L,f^*\mathbb L)$ such that $\ef_t:=i_t^*\ef$ satisfies $(f_t)_*\ef_t=e_t$, the derivative $\dot\gamma(t) \in C^\infty(\Gamma_t)$ is given by
\begin{equation}\label{eq:deriviso}
    \frac{d}{dt}(\Gamma_t,e_t)=-\frac{\i }{2\pi} (f_t)_*\left(\ef_t^{-1}\frac{D\ef}{dt}\Big|_{t}\right).
\end{equation}
\end{lm}
\begin{proof}
    Consider paths $\gamma$ with $\gamma(0)= (\Gamma_0,e_0)=(\Gamma,e)$. Let $\overline e$ be as in Definition~\ref{def:section}, which gives a trivialization $\cP|_\W\cong \W\times U(1)$ and thus an isomorphism 
\begin{equation}\label{eq:splitiso}
T_{(\Gamma,e)}\cP\cong \T{\Gamma}\oplus T_1U(1)\cong C^\infty (\Gamma),
\end{equation}
where we identify $T_1U(1)=\i\R$ with $\R$ by multiplication by $-\frac{\i}{2\pi}$. We show that the map given by formula~\eqref{eq:deriviso} is an isomorphism by showing that it coincides with the isomorphism~\eqref{eq:splitiso}.
We do this first for a special choice of representatives $f_t$ of $\Gamma_t$ and then show independence of the choice.

    Define a family of flat unitary sections on $\Gamma_t$ by $\overline e_t=\overline e|_{\U_{\Gamma_t}}$. Writing $e_t=a_t\overline e_t$ for $a_t\in U(1)$ with $a_0=1$, under the isomorphism~\eqref{eq:splitiso}, we have
    \begin{equation}\label{eq:gammadot}
        \dot\gamma(0)=\dot\Gamma_0-\frac{\i}{2\pi}\dot a_0 .
    \end{equation}

Let $h_0=\dot k_0$ be the derivative of $\Gamma_t$ at $t=0$ under the usual identification $T_\Gamma\OO\cong \T{\Gamma}$. Identify, as above, the Weinstein neighborhood $\Gamma \subset W \subset X$ and the open subset $V \subset T^*\Gamma,$ and recall the extension of $e$ to a section $\p$ of $\BL|_W$ parallel along fibers of $T^*\Gamma$ before Definition~\ref{def:section}. Assume that $f_t=dk_t\circ f_0$ for each $t$ and some $k_t\in \T{\Gamma}$ such that $\im dk_t \subset V.$ Define $\k\in C^\infty(I\times L)$ by $\k\circ i_t=k_t\circ f_0$, and
    \[
    \overline\ef=\exp\left(2\pi\i (\k+F(\Gamma_t))\right)f^*\p.
    \]
    Thus, $\overline\ef_t:=i_t^*\overline \ef=f_t^*\overline e_t$. By the definition \eqref{eq:zeta} it is clear that $dF_{\Gamma}=-\zeta_{\Gamma}=0$, so we get
    \begin{align*}
        \frac{D\overline \ef}{dt}\Big|_{t=0}&=\left(\frac{d}{dt}\Big|_{t=0}e^{2\pi \i (k_t\circ f_0+F(\Gamma_t))}\right)f_0^*\p\\
        &=2\pi\i f_0^*\left( \left(\dot k_0+dF_\Gamma\left(\dot k_0\right)\right)\p|_{\Gamma_0}\right)=2\pi \i f_0^*(h_0e_0),
    \end{align*}
    hence
    \begin{align*}
        \frac{D\ef}{dt}\Big|_{t=0}=\frac{D(a\overline \ef)}{dt}\Big|_{t=0}=\dot a_0\overline \ef_0+ \frac{D\overline \ef}{dt}\Big|_{t=0}=\left(\dot a_0+2\pi \i (h_0\circ f_0)\right)\ef_0,
    \end{align*}
    and finally
    \begin{align*}
        -\frac{\i }{2\pi}(f_0)_* \left(\ef_0^{-1}\frac{D\ef}{dt}\Big|_{t=0}\right)=h_0-\frac{\i}{2\pi}\dot a_0,
    \end{align*}
    which agrees with $\dot\gamma(0)$ under the isomorphism \eqref{eq:splitiso} according to \eqref{eq:gammadot} keeping in mind the definition of $h_0$.
    
    In general, $f_t$ is of the form $\hat f_t\circ\varphi_t$, for $\hat f_t$ as before and $\varphi_t\in\Diff(L)$. Let $\hat\ef\in C^\infty\left(I\times L,\hat f^*\mathbb L\right)$ such that $\hat\ef_t:=i_t^*\hat\ef$ satisfies $\big(\hat f_t\big)_*\hat\ef_t=e_t$.
    We get
    \begin{align*}
        \frac{D\ef}{dt}\Big|_{t=0}=\frac{D}{dt}(\varphi_t^*\hat\ef)\Big|_{t=0}=\nabla_{\dot\varphi_0}\hat \ef_0+\frac{D\hat \ef}{dt}\Big|_{t=0}\circ\varphi_0.
    \end{align*}
    Since $\hat \ef_0$ is parallel, the first term vanishes, and we get
    \[
    -\frac{\i }{2\pi}(f_0)_* \left(\ef_0^{-1}\frac{D\ef}{dt}\Big|_{t=0}\right)=-\frac{\i }{2\pi}\big(\hat f_0\big)_* \left(\hat \ef_0^{-1}\frac{D\hat \ef}{dt}\Big|_{t=0}\right).
    \]
\end{proof}
\begin{lm}\label{lm:connection}
    The horizontal subspace $H_\Gamma\subset T_{(\Gamma,e)}\cP$ associated with the flat connection is given by $\T{\Gamma}\subset C^\infty(\Gamma)$ under the identification of Lemma~\ref{lm:tanspce}.
\end{lm}
\begin{proof}
    Since $\T{\Gamma},H_\Gamma$ are both of codimension 1, it suffices to prove $H_\Gamma\subseteq \T{\Gamma}$. For this, it suffices to prove that the derivatives of the flat sections $\overline e:\W\to\cP $ in any direction lie in $\T{\Gamma}$ for arbitrary $\W$ and~$e$. Indeed, given a path $\Gamma_t=\Gamma_{k_t}\in \W$ for $k_t\in \T{\Gamma}$, by Definition~\ref{def:section} we have
    \[
    \overline e_{\Gamma_{k_t}}=\exp\left(2\pi \i (K_t+F(\Gamma_{k_t})\right)\p|_{\Gamma_{k_t}}
    \]
    and by the Lemma~\ref{lm:tanspce} together with $dF_{\Gamma}=-\zeta_\Gamma=0$, its derivative at $t=0$ is given by
    \[
    -\frac{\i}{2\pi}\cdot2\pi\i \dot k_0=\dot k_0\in\T{\Gamma}.
    \]
\end{proof}
\begin{cy}\label{cy:derivatives}
Under the identifications of Lemmas~\ref{lm:Otanspce} and~\ref{lm:tanspce}, the isomorphism $d\pi_{\cP}:H_{\Gamma}\to T_{\Gamma}\OO$ is given by the identity.
\end{cy}
\begin{proof}
    Under the identifications, $h,d\pi_\cP(h)$ differ by a constant for any $h\in T\cP$, and since elements of both $H_{\Gamma},T_{\Gamma}\OO$ are required to have zero integral, the result follows.
\end{proof}

Using Lemma~\ref{lm:tanspce} we make the following definition.
\begin{df}\label{df:metric}
    The Riemannian metric $\overline G$ on $\cP$ is given by
    \[
    \overline G(h,k)=\int_{\Gamma} hk\re\Omega, \qquad h,k\in T_{(\Gamma,e)}\cP\cong C^\infty(\Gamma).
    \]
\end{df}
\begin{rem}
    With the metric $G$ on $\OO$ and the metric $\overline G$ on $\cP$, the projection $\pi_\cP$ is a Riemannian submersion.
\end{rem}

\subsection{The derivative of \texorpdfstring{$\varphi_m$}{phim}}
For $h \in C^\infty(\Gamma)$ define $\chi_{\Gamma,e,h} \in \H_m^\vee$ by
\[
\chi_{\Gamma,e,h}(s) = -2\pi\i m\int_\Gamma h \frac{s}{e^m} \Omega.
\]
Keeping in mind the identification $T_{(\Gamma,e)}\cP\cong C^\infty(\Gamma)$ of Lemma~\ref{lm:tanspce}, we have:
\begin{pr}\label{pr:der}
The derivative of $\vp_m$ at $(\Gamma,e) \in \cP$ is given by
\[
d(\vp_m)_{(\Gamma,e)}(h) = \chi_{\Gamma,e,h}, \qquad h \in C^\infty(\Gamma).
\]
\end{pr}
\begin{proof}
Let $\gamma=(\Gamma_t,e_t)_{t\in (-\epsilon,\epsilon)} $ be a path with $\gamma(0)=(\Gamma,e)$ and $\dot\gamma(0)=h$. Write $h_t=\dot\gamma(t)$. Let $f,\ef$ be as in Lemma~\ref{lm:tanspce}.
Then we have
\[
d\vp_m(h) = \frac{d}{dt}\Big|_{t=0} \psi_{\Gamma_t,e_t} .
\]
So, we calculate
\begin{align*}
\frac{d}{dt} \psi_{\Gamma_t,e_t}(s)
&= \frac{d}{dt} \int_{\Gamma_t} \frac{s}{e_t^m} \Omega = \frac{d}{dt} \int_L f_t^* \left(\frac{s}{e_t^m} \Omega\right) \\
&= \int_L \frac{d}{dt} \ef_t^{-m} f_t^* (s \Omega).
\end{align*}
Since $f_t = f \circ i_t$ and $\frac{di_t}{dt} = \frac{\partial}{\partial t},$ by the Cartan formula,
\begin{align*}
   \frac{d}{dt} \ef_t^{-m} f_t^* (s \Omega) &=\frac{d}{dt}i_t^*\left(\ef^{-m}f^*(s\Omega) \right)\\
   &= i_t^* \left( i_{\frac{\partial}{\partial t}} d(\ef^{-m} f^*(s\Omega)) + d\left(i_{\frac{\partial}{\partial t}} \ef^{-m} f^*(s\Omega)\right)\right). 
\end{align*}
The second term does not contribute on account of Stokes' theorem, so we focus on the first. Since $s$ is holomorphic and $\Omega$ is of type $(n,0)$, it follows that $D(s\Omega) = 0.$ Thus,
\begin{align*}
d(\ef^{-m} f^*(s\Omega)) &= (D\ef^{-m})f^*(s\Omega) + \ef^{-m}f^* D(s \Omega) \\
&= -m\ef^{-m-1}(D\ef)f^*(s\Omega)
\end{align*}
and Lemma~\ref{lm:tanspce} together with the fact that $\ef_t=f_t^*e_t$ is parallel gives
\[\resizebox{\linewidth}{!}{$
\begin{aligned}
    i_t^* \Big( i_{\frac{\partial}{\partial t}} d(\ef^{-m} f^*(s\Omega))&\Big) =i_t^*\left(-m\ef^{-m-1}\left(\frac{D\ef}{dt}f^*(s\Omega)-(D\ef) i_{\frac{\partial}{\partial t}}f^*(s\Omega) \right)\right)\\
    &=-m\ef_t^{-m-1}\left(\frac{D\ef}{dt}\Big|_{t}f_t^*(s\Omega)-(D\ef_t)i_t^*\left(i_{\frac{\partial}{\partial t}}f^*(s\Omega) \right)\right)
    \\
    &= -2\pi\i m (f_t^*h_t) \ef_t^{-m} f_t^*(s\Omega).
\end{aligned}
$}\]
Finally,
\begin{align*}
    \frac{d}{dt}\Big|_{t=0} \psi_{\Gamma_t,e_t}(s)&=-2\pi\i m\int_L(f_0^*h) \ef_0^{-m} f_0^*(s\Omega)\\
    &=-2\pi\i m\int_\Gamma h\frac{s}{e^m}\Omega=\chi_{\Gamma,e,h}(s)
\end{align*}
\end{proof}

\begin{proof}[Proof of Theorem~\ref{tm:bound}]

Since $\re\Omega|_\Gamma$ agrees with the orientation, it satisfies $\re\Omega|_\Gamma=\left|\re\Omega|_\Gamma\right|_g\vol_\Gamma$, and by $\epsilon$-positivity
\[
\left|\Omega|_\Gamma\right|_g\vol_\Gamma\le\left(\left|\re\Omega|_\Gamma\right|_g+\left|\im\Omega|_\Gamma\right|_g\right)\vol_\Gamma\le\left(1+\frac1\epsilon\right)\re\Omega|_\Gamma.
\]

    In view of Proposition~\ref{pr:der}, we need to bound $\| \chi_{\Gamma,e,h}\| $. Let $s_1,...,s_k\in \HH_m$ be an orthonormal basis, and for each $j$ let $S_j=\max_X|s_j|_h$. 
    Note that $|s_j/e^m|=|s_j|_h$, since $e^m$ is unitary. Thus,
    \begin{align*}
        \left|\chi_{\Gamma,e,h}(s_j)\right|&=2\pi m\left|\int_\Gamma h\frac{s_j}{e^m}\Omega \right|\\
        &\le 2\pi m\int_\Gamma|h|\left|s_j\right|_h\left|\Omega\right|_g\vol_\Gamma\\
        &\le 2\pi m\left(1+\frac1\epsilon\right)\int_\Gamma|h||s_j|_h\re\Omega\\
        &\le2\pi m\left(1+\frac1\epsilon\right)\left(\int_\Gamma|h|^2\re\Omega\right)^{1/2}\left(\int_\Gamma|s_j|_h^2\re\Omega\right)^{1/2}\\
        &\le2\pi m\left(1+\frac1\epsilon\right)\|h\|_{\overline G} \, S_j(\re\Omega(A))^{1/2},
    \end{align*}
    where $A$ is the homology class represented by $\Gamma$. Note that $A$ is independent of the specific $\Gamma\in \OO$, as they are all homotopic.

  Putting $S=\left(S_1^2+...+S_k^2\right)^{1/2},$ we get
  \[
  \|d\varphi_m(h)\|\le 2\pi m\left(1+\frac1\epsilon\right)S(\re\Omega(A))^{1/2}\|h\|_{\overline G}
  \]
  so
  \begin{equation}\label{eq:dphi}
      \|d\varphi_m\|\le 2\pi m\left(1+\frac1\epsilon\right)S(\re\Omega(A))^{1/2}
  \end{equation}
    
\end{proof}
\begin{rem}
    Using Sobolev embeddings, elliptic regularity of $\overline\partial$ and Lemma~\ref{lm:fundamental}, it can be shown that the constants $S_j$ in the preceding proof are bounded by $Cm^{n+1}$ for some $C>0$ independent of $j,m$. Using the Hirzebruch-Riemann-Roch theorem to calculate $k = O(m^n)$, we can bound the constant $S$ in~\eqref{eq:dphi} by $Cm^{3n/2+1}$.
\end{rem}

\section{An approximation theorem}\label{sec:sep}
In this section we prove the following theorem, which is a self-contained restatement of Theorem~\ref{tm:approximation1} in its natural generality. In particular, there is no need for $X$ to be Calabi-Yau.
\begin{tm}\label{tm:approximation2}
Let $(X,\omega)$ be a compact K\"ahler manifold, and let $(L,h)$ be a holomorphic Hermitian line bundle over $X$ such that $F_h=-2\pi \i.$
Let $\Gamma\subset X$ be a compact Bohr-Sommerfeld Lagrangian submanifold, that is, $L|_\Gamma$ has trivial holonomy.
Let $e$ be a parallel unitary section of $L|_{\Gamma}$ and let $b\in C^\infty(\Gamma,\C) $. Let $V\subseteq X$ be open such that $\supp b\subseteq V\cap\Gamma.$ Then, for any $\varepsilon>0$ and $r\in \N,$ there is $M\in\N$ such that for $m\ge M$ there is a holomorphic section $s\in H^0(X,L^m)$ such that
    \[
    \left\|s|_\Gamma-be^m\right\|_{C^r}<\varepsilon, \qquad
    \left\|s|_{X\setminus V}\right\|_{C^r}<\varepsilon.
    \]
\end{tm}
Versions of this theorem appear in the previous works~\cite[Proposition~2.4]{Char03}, \cite[Proposition~19]{Ve20}, and~\cite[Proposition~2]{Ti25}.
\begin{rem}
    The $C^r$-norm in this theorem is defined in terms of the covariant derivative and the metric. There is another $C^r$-norm, defined in terms of coordinate charts and the standard $C^r$-norm on functions on Euclidean space. These two norms are equivalent up to constants. See the discussion before Lemma~\ref{lm:Cr} and Lemma~\ref{lm:Cr} itself, which analyzes the asymptotic growth of the constants as $m\to\infty$.
\end{rem}
In the proof of Theorem~\ref{tm:approximation2} we take inspiration from the proofs of Tian~\cite[Lemma~1.2]{tian1990polarized}, Nirenberg-Wells~\cite[Theorem~6.1]{NiW69} and Duval~\cite[Lemma]{Duval2015}: We start with an almost-holomorphic section $v$ extending $be^m$, and then perturb it to obtain a holomorphic section $s=v-u$, where $u$ is a small solution to the equation $\overline\partial u=\overline\partial v$.
\subsection{Preliminary lemmas}\label{subsec:lemmas}
For a hermitian vector bundle $(E,h)$ over a Hermitian manifold $(X,g)$, $(p,q)$-forms with values in $E$ are equipped with the $L^2$-norm
\[
\|w\|^2=\int _X|w|_{h,g}^2\vol_g.
\]
\begin{df}\label{df:positivity}
    Let $(M,J)$ be an almost complex manifold. A 2-form $\alpha$ is said to be positive if for all $x\in M,v\in T_xM$ with $v\ne 0$ we have
    \[
    \alpha(v,Jv)>0.
    \]
    It is said to be non-negative if a weak inequality holds for all $x,v$. We write $\alpha>\beta$ or $\alpha\ge \beta$ if $\alpha-\beta$ is positive or non-negative respectively.
\end{df}
The following is a standard result using H\"ormander's $L^2$ method:
\begin{lm}\label{lm:3.1}
    Let $(L,h)$ be a Hermitian holomorphic line bundle over a compact K\"ahler manifold $(X,\omega)$, and assume there is a constant $C>0$ such that $\i F_h+\i \ric g\ge C\omega$. 
 Then for any $w\in A^{0,1}(L)$ with $\overline\partial w=0$ there is $u\in A^0(L)$ such that $\overline\partial u=w$ and
    \[
    \|u\|^2\le \frac{1}{C}\|w\|^2
    \]
\end{lm}
\begin{proof}
    Define $\tilde L=L\otimes K_X^* = L\otimes \Lambda^n T_{1,0}X$. Then $\tilde L$ is holomorphic, and carries a Hermitian metric $\tilde h=h\otimes \det g$, the curvature of which is
    \[
    F_{\tilde h}=F_h+ F_{ \det g}=F_h+ \ric g
    \]
    so by assumption, $\i F_{\tilde h}\ge C\omega$. By~\cite[Theorem~8.8]{Demailly1996}, a similar assertion to ours holds for $(n,1)$-forms with values in $\tilde L$. 

    But note that $(n,1)$-forms with values in $\tilde L$ are the same as $(0,1)$-forms with values in
    \[
    \tilde L\otimes \Lambda^nT^{1,0}X =L\otimes \Lambda^nT^{1,0}X\otimes \Lambda^nT_{1,0}X\cong L
    \]
    where the isomorphism preserves the metric, so our claim follows.
\end{proof}
Let $S$ be an $m$-dimensional submanifold of an $n$-fold $M$. We call a coordinate chart $\phi:U\to \R^n=\R^m\times\R^{n-m}$ for $M$ a \textit{slice chart} for $S$, or \textit{slice coordinates} for $S$, if $\phi^{-1}(\R^m\times\{0\})=S\cap U$.
\begin{lm}\label{lm:rho}

    There is a function $\rho$ in a neighborhood of $S$ such that $\rho^2$ is smooth and for each slice coordinate chart $\phi: U\to \R^m\times \R^{n-m},$
    denoting by $x,y$ the coordinates on $\R^m,\R^{n-m}$ respectively, we have  $\rho\le\mu|y|$ and $|y|\le\nu\rho$ in $U$, for some continuous functions $\mu,\nu\ge0$.
\end{lm}
\begin{proof}
    Take a locally finite cover of $S$ by slice coordinate charts $ U_\alpha$. Assume each $U_\alpha$ is relatively compact in a slice coordinate chart $O_\alpha$, such that $\{O_\alpha\}$ is also locally finite.
    On each $ U_\alpha$ define $\rho_\alpha^2=|Y_\alpha|^2$ where $(X_\alpha,Y_\alpha)$ are the coordinates on $U_\alpha$. Then take a partition of unity $\{\phi_\alpha\}$ subordinate to $\{U_\alpha\}$, and define $\rho^2=\sum\phi_\alpha\rho_\alpha^2 $.

    Now let $ U$ be any slice chart. The transition maps $\varphi_{\alpha}$ between $ U, O_\alpha$ are of the form
$$\varphi_{\alpha}(x,y)=(X_\alpha(x,y),Y_\alpha(x,y))$$
where $Y_\alpha(x,0)=0$, so $Y_\alpha=\mu_\alpha y$ for some continuous matrix-valued function $\mu_\alpha$ on $U_\alpha\cap U$. By symmetry also $y=\nu_\alpha Y_\alpha$. This gives
\[
\rho^2=\sum \phi_\alpha \left|Y_\alpha\right|^2\le\sum\phi_\alpha |\mu_\alpha|^2|y|^2=\mu^2|y|^2
\]
where $\mu^2:=\sum\phi_\alpha|\mu_\alpha|^2$. Also, taking cutoff functions $\psi_\alpha$ such that $\psi_\alpha|_{U_\alpha} \equiv 1$ and $\supp\psi_\alpha \subset O_\alpha$, we have
\[
|y|^2=\sum\phi_\alpha|y|^2\le\sum\phi_\alpha|\nu_\alpha|^2|Y_\alpha|^2=\sum\phi_\alpha\psi_\alpha^2|\nu_\alpha|^2|Y_\alpha|^2
\]
and defining $\nu=\sup_\alpha\psi_\alpha|\nu_\alpha|$, we get
\[
|y|^2\le \sum\phi_\alpha\nu^2|Y_\alpha|^2=\nu^2\rho^2.
\]
Observe that $\nu$ is continuous by local finiteness of $\{O_\alpha\}$ and the fact that $\psi_\alpha|\nu_\alpha|$ is continuous on $U$ and supported in $U\cap O_\alpha$.
\end{proof}
\begin{rem}\label{rem:rho}
     In particular, if $S$ is compact and we cover it by a finite collection of slice coordinate charts $U_j$, assuming each of them is relatively compact in a larger coordinate chart, we get a constant $B>0$ such that
    \[
B^{-1}|y|^2\le \rho^2(x,y)\le B|y|^2
\]
in each $U_j$.
\end{rem}
\begin{df}\label{df:Zr}
    For a neighborhood $U$ of $S$ and $r>0$, define
    \[
    Z_{r,U}=\{z\in U:\rho(z)<r\}.
    \]
    when the set $U$ is clear from the context, we will write simply $Z_r$.
\end{df}
\begin{lm}\label{lm:Zrclosed}
    If $S$ is compact then we can find a neighborhood $U$ of $S$ such that for small enough $r$, $\overline Z_{r}\subset U$.
\end{lm}
\begin{proof}
    Let $O$ be the domain of definition of $\rho$, and take $U$ to be relatively compact in $O$. Let also $V$ be a neighborhood of $S$ with $\overline V\subset U$. Then $\rho$ does not vanish on the compact set $\overline U\setminus V$, so it has a minimum $R$ there.
    Then for $r<R$ we have $Z_r\subseteq V$, so $\overline Z_r\subseteq \overline V\subset U$.
\end{proof}
\begin{lm}\label{lm:Jcoor}
    Let $L$ be a totally real submanifold of an almost complex manifold $\left(M,J\right)$.
Then locally around each point $p\in L$, there is a smooth coordinate
chart $\phi:U\to M$ where $U\subseteq\mathbb{C}^{n}$, such that
$\phi^{-1}(L)=U\cap\mathbb{R}^{n}$ and $\left(\phi^{*}J\right)|_{\mathbb{R}^{n}}=J_{0}$, where $J_0$
is the standard almost complex structure on $\mathbb{C}^{n}$.
Moreover, such coordinates can be found to extend any given coordinates
$\psi:V\to L$ where $V\subseteq\mathbb{R}^{n}$.
\end{lm}
\begin{proof}
Take a Riemannian metric $g$ on $M$. Use the exponential map $\exp:JTL\to M$ to define a tubular neighborhood $U$ of $L$, as the image under $\exp$ of $N_\epsilon=\{u\in JTL:|u|<\epsilon(\pi(u))\}$ for a small enough continuous map $\epsilon:L\to\R_{>0}$.

We define an almost complex structure $J_\epsilon=\exp^*J$ on $N_\epsilon$, for which $L$ (viewed as the zero section) is totally real. Under the identification $TN_\epsilon|_L\cong TL\oplus JTL=TM|_L$, it is known that $d\exp_{(x,0)}:T_{(x,0)}N_\epsilon\to T_xM $ is the identity, so for $x\in L$ we have
\[
J_\epsilon|_{(x,0)}=\left(d\exp_{(x,0)}\right)^{-1}J_xd\exp_{(x,0)}=J_x.
\]

Now, the coordinates $\psi$ on $L$ induce coordinates $\varphi:V\times\R^n\to JTL $ by
\[
\varphi(x,y)=\left(\psi(x),J_{\psi(x)}d\psi_xy\right)
\]
and along $V\times\{0\}$ the differential is $d\varphi_{(x,0)}=d\psi_x\oplus J_{\psi(x)} d\psi_x $. Hence we have
\begin{align*}
    J_\epsilon|_{\varphi(x,0)} d\varphi_{(x,0)}\frac{\partial}{\partial x^i}=J_{\psi(x)} d\psi_x \frac{\partial}{\partial x^i}=d\varphi_{(x,0)}\frac{\partial}{\partial y^i}
\end{align*}
so $\varphi^*J_\epsilon=J_0$ along $V\times\{0\}$. Hence the desired coordinates are $\phi:=\exp\circ \varphi$ on $U=\varphi^{-1}(N_\epsilon)\subseteq \C^n$.
\end{proof}
If $S\subseteq M$ is a submanifold, then a function $f$ on $M$ is said
to vanish to $N$-th order on $S$ if in any local coordinates
we have $\frac{\partial^\alpha}{\partial x^\alpha}f=0$ on $S$ for any $\left|\alpha\right|\le N$,
or equivalently $f=O\left(|y|^{N+1}\right)$ in slice coordinates, which is the same as $f=O(\rho^{N+1})$ for the function $\rho$ given by Lemma~\ref{lm:rho}. 

A section $u$ of a vector bundle $E$ is said to vanish along $S$ to $N$-th order if when written as $u=\sum u^js_j$ for a local frame $s=(s_1,...,s_k)$, the coefficient functions $u^1,...,u^k$ vanish to $N$-th order along $S$. It is enough to check on just one local frame around each point, as the property of vanishing to $N$-th order is preserved by multiplying by a smooth function. We denote the coefficient vector $\left(u^1,...,u^k\right)$ by $us^{-1}$.

The following lemma is a small modification of a result of Nirenberg-Wells~\cite{NiW69}.
\begin{lm}\label{lm:FunExt}
    Let $X$ be a complex manifold and let $\Gamma\subset X$ be a totally real submanifold which is a closed subset. Then for every smooth function $f:\Gamma\to\C$ and every $N\in \N$ there is an extension $\tilde f:X\to \C$ such that $\overline\partial \tilde f$ vanishes to $N$-th order along $\Gamma$.

    If $V\subseteq X$ is open such that $\supp f\subseteq V\cap \Gamma$, then we can assume $\supp\tilde f\subseteq V$.
\end{lm}
\begin{proof}
    Start with the case where there is a global slice chart on $X$ of the type introduced by Lemma~\ref{lm:Jcoor}. 
    By applying \cite[Lemma 3.1]{NiW69} with the defining functions the imaginary coordinates $y^1,...,y^n$, there is an extension $\hat f$ to a neighborhood $U$ of $\Gamma$ in $X$ with $\overline\partial\hat f$ vanishing to $N$-th order. Multiplying by a suitable cut-off function gives a global extension.
    
We modify $\hat f$ to obtain $\tilde f$ with the desired support. Assume $\supp f\subseteq V\cap \Gamma$, and let $\hat f$ be any extension with $\overline\partial \hat f$ vanishing to $N$-th order.
    Let $0\le \varphi\le 1$ be a smooth bump function supported in $V$ which equals $1$ on a neighborhood of $\supp f$ in $V$, and define $\tilde f=\varphi\hat f$. Then $\tilde f$ is an extension of $f$ which is supported in $V$. We claim that $\overline\partial \tilde f$ vanishes to $N$-th order along $\Gamma$. Indeed, we have
\[
\overline\partial {\tilde f}=\overline\partial \varphi\cdot \hat f+\varphi\overline\partial\hat f.
\]
The second summand vanishes to $N$-th order along $\Gamma$, so we need to check for the first. Since $\varphi\equiv1$ on a neighborhood of $\supp f$, it is left to check that the term vanishes to $N$-th order along $\Gamma\setminus\supp f$. It will suffice to show that $\hat f$ vanishes to $N$-th order there.

Indeed, $\hat f$ vanishes to $0$-th order and $\overline \partial \hat f$ vanishes to $N$-th order along $\Gamma\setminus\supp f$. Working in the smooth slice coordinates such that $J=J_0$ along $\Gamma$, we have
\[
\overline\partial\hat f=\frac{1}{2}\frac{\partial \hat f}{\partial x^i}\left(dx^i+\i Jdx^i\right)+\frac{1}{2}\frac{\partial \hat f}{\partial y^i}\left(dy^i+\i Jdy^i\right).
\]
Along $\Gamma$ the vectors $dy^i+\i Jdy^i=dy^i+\i dx^i$ form a basis for $T^{0,1}X$, so the vectors $dy^i+\i Jdy^i$ form a basis also on a neighborhood of $\Gamma$. Hence their coefficients in $\overline\partial\hat f$ vanish to $N$-th order along $\Gamma$. We can write
\[
dx^j+\i Jdx^j=a_i^j\left(dy^i+\i Jdy^i\right),
\]
and defining $c_i=\frac{1}{2}\frac{\partial \hat f}{\partial y^i}+ a_i^j\frac{\partial \hat f}{\partial x^j}$ we have
\[
\overline\partial\hat f=c_i\left(dy^i+\i Jdy^i\right),
\]
so $c_i$ vanishes to $N$-th order along $\Gamma$ for each $i$.

Now we prove that a derivative of $\hat f$ up to order $N$ vanishes along $\Gamma\setminus\supp f$. The proof proceeds by induction on the number $k$ of derivatives in the $y$-directions.
For $k=0$ this follows from the fact that $\hat f\equiv0$ on $\Gamma\setminus\supp f$, so differentiating using only $x$-coordinates gives zero.
Suppose the claim is true for $0\le k<\ell\le N$, let $\alpha=(\alpha_1,...,\alpha_n)$ be a multi-index with $|\alpha|=\ell$, and let $\beta$ be any multi-index with $m=|\beta|\le N-\ell$. Without loss of generality we may assume $\alpha_1>0$, and define $\alpha'_i=\alpha_i-\delta_{i1}$. Then
\begin{align*}
    \frac{\partial^{m+\ell}\hat f}{\partial x^\beta\partial y^\alpha}=\frac{\partial^{m+\ell-1}}{\partial x^\beta\partial y^{\alpha'}}\frac{\partial\hat f}{\partial y^1}=\frac{\partial^{m+\ell-1}}{\partial x^\beta\partial y^{\alpha'}}c_1-\frac{\partial^{m+\ell-1}}{\partial x^\beta\partial y^{\alpha'}}\left(a_1^j\frac{\partial \hat f}{\partial x^j}\right).
\end{align*}
The first summand vanishes along $\Gamma\setminus\supp f$ since $c_1$ vanishes to $N$-th order along $\Gamma$. Applying the Leibniz rule to the second summand, each term contains a derivative of $\hat f$ of order up to $N$ with less than $\ell$ derivatives in the $y$-directions. Hence it vanishes along $\Gamma\setminus\supp f$ by the induction hypothesis.

For the general case, cover $\Gamma$ by slice coordinate charts $\{U_\alpha\}$ of the type introduced by Lemma~\ref{lm:Jcoor}, and let $\{\varphi_\alpha\}$ be a partition of unity subordinate to that cover. In each $U_\alpha$, let $\tilde f_\alpha$ be an extension of $\varphi_\alpha f$ supported in $\supp \varphi_\alpha\cap V$ with $\overline\partial \tilde f_\alpha$ vanishing to $N$-th order. Then $\tilde f:=\sum \tilde f_\alpha$ is defined in $U=\bigcup U_\alpha$ and supported in $V,$ and $\overline\partial \tilde f$ vanishes to $N$-th order. Multiplying by a suitable cut-off function gives a globally defined extension with the desired properties.

\end{proof}
If $f$ is a function on a manifold $M$ and $S\subset M$ is a critical submanifold - that is, $df$ vanishes along $S$ - then the normal Hessian is a canonical map $\hess f:NS\otimes NS\to \underline{\R} $ where $NS=TM|_S/TS$ is the normal bundle.
\begin{lm}\label{lm:BSextpara}
    Let $E$ be a holomorphic vector bundle over a complex manifold $X$,
and let $\Gamma\subset X$ be a totally real submanifold which is a closed subset. Let $\xi$
be a section of $E$ over $\Gamma$. Then for every $N$ there is
a section $\tilde{\xi}$ of $E$ over $X$,
such that $\overline{\partial}\tilde{\xi}$ vanishes along $\Gamma$
to $N$-th order. 
If $V\subseteq X$ is open such that $\supp \xi\subset V\cap\Gamma$, then we may assume $\supp\tilde \xi\subset V$.

Suppose moreover that $E$ is equipped with a Hermitian metric $h$, such that $\i F_h$
is positive along $\Gamma$ and $\xi$ is unitary parallel with respect to the Chern connection. If $N\ge 1$ then $\Gamma$ is a critical submanifold for $a:=\big|\tilde{\xi}\big|_h$
and the normal Hessian $\hess a$ is negative-definite at each point. In particular $a\le 1$ in a neighborhood of $\Gamma$.
\end{lm}
\begin{proof}
    Let $s_{\alpha}=(s_{\alpha,1},...,s_{\alpha,k})$ be a holomorphic frame in $U_{\alpha}$, for a locally finite family
$\left\{ U_{\alpha}\right\} $ that covers $\Gamma$, and let $\left\{ \phi_{\alpha}\right\} $
be a partition of unity subordinate to $\left\{ U_{\alpha}\right\} $.
Define $\xi_{\alpha}=\phi_{\alpha}\xi$ and $f_\alpha=\xi_\alpha s_{\alpha}^{-1}:\Gamma\to \C^k$, that is, $\xi_\alpha=f_\alpha s_{\alpha}=\sum f_\alpha^js_{\alpha,j}$. Then $\supp f_\alpha\subset U_\alpha\cap\Gamma$.

Let us assume we are given an open set $V$ with $\supp \xi\subseteq V\cap \Gamma$, possibly $V=X$. Then $\supp f_\alpha\subseteq V\cap U_\alpha\cap\Gamma$.
By Lemma~\ref{lm:FunExt} we can extend $f_{\alpha}$
to a function $\tilde{f}_{\alpha}$ supported in $V\cap U_\alpha$
such that $\overline{\partial}\tilde{f}_{\alpha}$ vanishes along
$\Gamma$ to $N$-th order.

Define $\tilde{\xi}_{\alpha}=\tilde{f}_{\alpha}s_{\alpha}=\sum \tilde f_\alpha^js_{\alpha,j}$,
and define $\tilde{\xi}=\sum\tilde{\xi}_{\alpha}$. It follows that $\overline\partial{\tilde\xi}_\alpha=\overline\partial\tilde f_\alpha \cdot s_\alpha$ vanishes to $N$-th order, hence so does $\overline\partial\tilde\xi=\sum\overline\partial\tilde\xi_\alpha$. Also, since each $\tilde \xi_\alpha$ is supported in $V\cap U_\alpha\subseteq V$, it follows that $\tilde \xi$ is supported in $V$.

Now assume that $\xi$ is unitary and parallel with respect to some positive Hermitian metric on $E$. To prove the assertion, it suffices to show that $a^2$
has a critical submanifold at $\Gamma$ with negative-definite normal Hessian. For
this we compute the derivative and the normal Hessian using fibers of $JT\Gamma$ as representatives for $N\Gamma$.

Let $0\ne \theta\in T_{p}\Gamma$ at some point $p\in\Gamma$. First we show how to extend $\theta$ locally to a vector field on $X$, also denoted $\theta$, such that
\begin{equation}\label{eq:theta}
    [\theta,J\theta]|_\Gamma=0.
\end{equation}
Find coordinates for $\Gamma$ in which $\theta=\partial/\partial x^{1}|_p$, then use Lemma~\ref{lm:Jcoor} to extend these coordinates to $\mathbb{C}^{n}$-coordinates
on $X$, such that $J=J_{0}$ on $\Gamma$. Extend $\theta$ to a
neighborhood of $p$ in $X$ by $\theta=\partial/\partial x^{1}$,
and then $J\theta=\partial/\partial y^{1}$ along $\Gamma$.
For $f$ a function on $X$, and for any $q\in \Gamma $ where $\theta$ is defined, take a path $\gamma$ in $\Gamma$ with $\gamma(0)=q,\dot\gamma(0)=\partial/\partial x^1$. Then
\begin{align*}
    \left[\theta,J\theta\right]_qf&=\frac{\partial}{\partial x^{1}}\Big|_{q}((J\theta) f)-(J\theta)_q\frac{\partial}{\partial x^{1}}f\\
    &=\frac{d}{dt}\Big|_0(J\theta)_{\gamma(t)}f-\frac{\partial}{\partial y^1}\Big|_q\frac{\partial}{\partial x^1}f\\
    &=\frac{\partial}{\partial x^1}\Big|_q\frac{\partial}{\partial y^1}f-\frac{\partial}{\partial y^1}\Big|_q\frac{\partial}{\partial x^1}f=0,
\end{align*}
proving \eqref{eq:theta}.

Using the Chern connection on $E,$ we calculate,
\[
d\left(a^{2}\right)(J\theta)=\left\langle D_{J\theta}\tilde{\xi},\tilde{\xi}\right\rangle +\left\langle \tilde{\xi},D_{J\theta}\tilde{\xi}\right\rangle. 
\]
Decompose by type $\theta=\theta_{1,0}+\theta_{0,1}.$ Then 
\[
J\theta=\i\theta_{1,0}-\i\theta_{0,1}=\i\theta-2\i\theta_{0,1},
\]
and we have
\[
D_{J\theta}\tilde{\xi}=\i D_{\theta}\tilde{\xi}-2\i D_{\theta_{0,1}}\tilde{\xi}.
\]
But $D_{\theta_{0,1}}\tilde{\xi}|_\Gamma=D_{\theta_{0,1}}^{0,1}\tilde{\xi}=\overline{\partial}\tilde{\xi}(\theta_{0,1})=0$ along $\Gamma$ by definition of $\tilde\xi$, so
\begin{equation}\label{eq:DJthetaxi}
    \left(D_{J\theta}\tilde{\xi}\right)\Big|_{\Gamma}=\left(\i D_{\theta}\tilde{\xi}\right)\Big|_{\Gamma}=\i D_{\theta}\xi=0
\end{equation}
with the last equality because $\xi$ is parallel. Since $\theta$ was arbitrary and $a$ is constant on $\Gamma$ it follows that $d\left(a^2\right)=0$, so $\Gamma$ is a critical
submanifold for $a^{2}$.

Now we show that the normal Hessian is negative: Let $[J\theta]$ denote the equivalence class of $J\theta$ considered as a section of the quotient bundle $N\Gamma.$ First,
\begin{equation}\label{eq:Hess}
\mathclap{
    \begin{aligned}
\hess\left(a^2\right)\left([J\theta],[J\theta]\right)&=J\theta\left(J\theta\left(a^2\right)\right)\\
&   =J\theta\left(\left\langle D_{J\theta}\tilde{\xi},\tilde{\xi}\right\rangle +\left\langle \tilde{\xi},D_{J\theta}\tilde{\xi}\right\rangle \right)\\
 & =\left\langle D_{J\theta}D_{J\theta}\tilde{\xi},\tilde{\xi}\right\rangle +2\left\Vert D_{J\theta}\tilde{\xi}\right\Vert ^{2}+\left\langle \tilde{\xi},D_{J\theta}D_{J\theta}\tilde{\xi}\right\rangle\\
 & \overset{\mathclap{\eqref{eq:DJthetaxi}}}{=}2\re \left\langle D_{J\theta}D_{J\theta}\tilde{\xi},\tilde{\xi}\right\rangle
\end{aligned}}
\end{equation}
Now,
\[
D_{J\theta}\tilde{\xi}=\i D_{\theta}\tilde{\xi}-2\i D_{\theta_{0,1}}\tilde{\xi}=\i D_{\theta}\tilde{\xi}-2\i\overline{\partial}\tilde{\xi}(\theta_{0,1})
\]
and since $\overline{\partial}\tilde{\xi}$ vanishes to first order,
$D_{J\theta}\left(\overline{\partial}\tilde{\xi}(\theta_{0,1})\right)=0$
along $\Gamma$. Thus, along $\Gamma$ we get 
\begin{align*}
D_{J\theta}D_{J\theta}\tilde{\xi} & =D_{J\theta}\i D_{\theta}\tilde{\xi}=\i D_{J\theta}D_{\theta}\tilde{\xi}\\
 & =\i D_{\theta}D_{J\theta}\tilde{\xi}+\i D_{[J\theta,\theta]}\tilde{\xi}+\i F_h\left(J\theta,\theta\right)\tilde{\xi}\\
 & \overset{\mathclap{\eqref{eq:theta},\eqref{eq:DJthetaxi}}}{=}-\i F_h\left(\theta,J\theta\right)\xi.
\end{align*}
By positivity of $\i F_h$ and \eqref{eq:Hess} we finally get
\begin{align*}
    \hess\left(a^2\right)\left([J\theta],[J\theta]\right)&=-2\re\left\langle\i \left(F_h\right)\left(\theta,J\theta\right)\xi,\xi\right\rangle<0.
\end{align*}
\end{proof}
Denote $R_m=\frac{\log m}{\sqrt m}$. Then we can achieve the following super-polynomial bound on $a$ as in the work of Tian~\cite[Lemma 1.2]{tian1990polarized}:
\begin{lm}\label{lm:a}
Under the conditions of the previous lemma and assuming in addition that $\Gamma$ is compact, we can find a neighborhood $U$ of $\Gamma$ such that $a^m\le m^{-\gamma\log m}$ in $Z_{2R_m,U}\setminus Z_{R_m,U}$, for $m$ large enough and some $\gamma>0$.
\end{lm}
\begin{proof}
   Take a cover $\{U_j\}$ of $\Gamma$ and a constant $B > 0$ as in Remark~\ref{rem:rho}, and define $U=\bigcup U_j$.
   By Lemma~\ref{lm:BSextpara}, $a$ has a strong local maximum along $\Gamma$. Let $\gamma >0$ be such that $-4\gamma B<0$ is the maximum value of the largest eigenvalue of $\hess a$ along $\Gamma$. Then for each $j$ and $z=(x,y)\in U_j$ we have 
\begin{align*}
a(z)&=1+\frac{1}{2}\left(\hess a\right)_{x}([y],[y])+O\left(|y|^{3}\right)\\
&\le1-2\gamma B|y|^{2}+O\left(|y|^{3}\right)\\
&\le1-\gamma B|y|^{2}\le 1-\gamma \rho^2(z)
\end{align*}
for small $y$, or equivalently, small $\rho$. Hence, for $z\in Z_{2R_m}\setminus Z_{R_m} $ we have
\begin{align*}
    a^m&\le \left(1-\gamma R_{m}^{2}\right)^{m}=\exp\left(m\log\left(1-\gamma R_{m}^{2}\right)\right)\\
    &\le\exp\left(-\gamma mR_{m}^{2}\right)=\exp\left(-\gamma\log^{2}m\right)\\
    &=m^{-\gamma\log m}
\end{align*}
\end{proof} 
Next we need two lemmas about the asymptotic rate of growth of certain constants. For the first, recall that if $\nabla$ is a connection on a vector bundle $E$ over a Riemannian manifold $(M,g)$, then the $k$-th covariant derivative $\nabla^{(k)}$ is defined inductively by $\nabla^{(k+1)}=\left(\nabla\otimes\tilde\nabla^{\otimes k}\right)\circ\nabla^{(k)}$ where $\tilde\nabla$ is the induced Levi-Civita connection on $T^*M$.

Recall also that if $E$ is a vector bundle equipped with a metric $h$ and any connection $\nabla$ over a compact Riemannian manifold $(M,g)$, then the $C^r$-topology on sections of $E$ can be defined either using the norm
\[
\|u\|_{C^r_{\text{int}}}=\sum_{i=0}^r \sup_M \left|\nabla^{(i)} u\right|_{h,g}
\]
or using a finite cover by coordinate charts $U_j$ with orthonormal frames $s_j$ on them, and by the norm
\[
\|u\|_{C^r_{\text{ext}}}=\max_j\left\|us_j^{-1} \right\|_{C^r(U_j)}=\max_j\sum_{|\alpha|\le r}\sup_{U_j}\left|\frac{\partial^\alpha}{\partial x^\alpha} \left(us_j^{-1}\right)\right|.
\]
Then we have:
\begin{lm}\label{lm:Cr}
    Let $E \to M$ be a vector bundle of rank $p.$ Assume $s_j$ are the restrictions of frames on coordinate charts $O_j$ such that $\overline U_j\subset O_j$. Then there is a constant $c>0$ such that if $u$ is a section of $E^{\otimes m}$, then
    \[
    \left(c(mp^m)^r\right)^{-1}\|u\|_{C^r_{\mathrm{ext}}}\le\|u\|_{C^r_{\mathrm{int}}}\le c(mp^m)^r\|u\|_{C^r_{\mathrm{ext}}}.
    \]
\end{lm}
\begin{proof}
    We will prove a bit more, namely, that there are constants $c(r,\ell),c'(r,\ell)$ such that the two bounds hold for the norms on sections of the bundle $E^{\otimes m}\otimes (T^*M)^{\otimes \ell}$ with respect to the metric $h^{\otimes m}\otimes (g^{-1})^{\otimes \ell}$ and the connection $\nabla^{m,\ell}:=\nabla^{\otimes m}\otimes \tilde\nabla^{\otimes \ell}$. 
    
    Let $\varepsilon_j$ be an orthonormal frame for $T^*M$ on $O_j$.
    Let $\theta_j,\kappa_j$ be the matrix-valued 1-forms on $O_j$ such that the local representation of the connections is
    \begin{align*}
        \nabla u\cdot s_j^{-1}&=d\left(us_j^{-1}\right)+us_j^{-1}\theta_j,\\
        \tilde\nabla\rho\cdot\varepsilon_j^{-1}&=d\left(\rho\varepsilon_j^{-1}\right)+\rho\varepsilon_j^{-1}\kappa_j.
    \end{align*}
    Then $\theta_j\varepsilon_j^{-1},\kappa_j\varepsilon_j^{-1}$, which are the application of $\varepsilon_j^{-1}$ to each entry of $\theta_j,\kappa_j$, are $3$-index tensors. They are bounded on $U_j$ in the Euclidean $C^r$-norm for any~$r$. Let $n=\dim M$, and denote by
    \[
    \theta_{j,m,\ell}=\sum_{a +b = m-1} I_{p}^{\otimes a} \otimes \theta_j \otimes I_{p}^{\otimes b}\otimes I_{n^\ell}
     +\sum_{a+b=\ell-1} I_{p^m}\otimes I_{n}^{\otimes a}\otimes \kappa_j\otimes I_{n}^{\otimes b} 
    \]
    the matrix-valued 1-form representing $\nabla^{m,\ell}$ on $U_j$. Then,
    \begin{equation}\label{eq:thetajmell}
    \mathclap{
      \begin{aligned}
          \|\theta_{j,m,\ell}&\varepsilon_j^{-1} \|_{C^r} \leq 
    mp^{m-1}n^\ell\|\theta_j\varepsilon_j^{-1}\|_{C^r}+p^m\ell n^{\ell-1}\|\kappa_j\varepsilon_j^{-1}\|_{C^r} \\
    &\le m p^{m}\left(n^\ell\|\theta_j\varepsilon_j^{-1}\|_{C^r}+\ell n^{\ell-1}\|\kappa_j\varepsilon_j^{-1}\|_{C^r}\right)=:mp^{m}C_{j,r,\ell}.
      \end{aligned}}      
    \end{equation}
    
    Now we prove the claim by induction on $r$. In the case $r=0$ there is an equality of the two norms, and we denote them both by~$\|\cdot\|_{C^0}$.
    Assume the claim is true for $r$. 
    There is a constant $c_{n,r}$ depending only on $n$ and $r$ such that the standard $C^r$-norm on functions on $n$-dimensional Euclidean space satisfies
    \[
    \|FG\|_{C^r}\le c_{n,r}\|F\|_{C^r}\|G\|_{C^r}
    \]
    for any matrix-valued $F,G$, regardless of their dimensions.
    To prove the claim for $r+1$ we calculate for a section $u$ of $E^{\otimes m}\otimes (T^*M)^{\otimes \ell}$,
    \[\resizebox{\linewidth}{!}{$
    \begin{aligned}
       \left\|\nabla^{m,\ell}\right.&\left.\!\!u\right\|_{C^r_{\text{int}}}\le c(r,\ell+1)(mp^m)^r\max_j\left\|\nabla^{m,\ell}u\cdot \left(s_j^{\otimes m}\otimes\varepsilon_j^{\otimes(\ell+1)}\right)^{-1}\right\|_{C^r(U_j)} \\
        &=c(r,\ell+1)(mp^m)^r\max_j\left\|d\left(u\left(s_j^{\otimes m}\otimes\varepsilon_j^{\otimes \ell}\right)^{-1}\right)\varepsilon_j^{-1}+\right.\\
        &\left.\qquad\qquad\qquad\qquad\qquad\quad\ +\ \!u\left(s_j^{\otimes m}\otimes\varepsilon_j^{\otimes \ell}\right)^{-1}\theta_{j,m,\ell}\varepsilon_j^{-1}\right\|_{C^r(U_j)}\\
        &=c(r,\ell+1)(mp^m)^r\max_j\left\|\sum_i\frac{\partial}{\partial x^i}\left(u\left(s_j^{\otimes m}\otimes\varepsilon_j^{\otimes \ell}\right)^{-1}\right)dx^i\varepsilon_j^{-1}+\right.\\
        &\left.\qquad\qquad\qquad\qquad\qquad\quad\ +\ \!u\left(s_j^{\otimes m}\otimes\varepsilon_j^{\otimes \ell}\right)^{-1}\theta_{j,m,\ell}\varepsilon_j^{-1}\right\|_{C^r(U_j)}\\
        &\overset{\mathclap{\eqref{eq:thetajmell}}}{\le} c_{n,r}c(r,\ell+1) (mp^m)^r\max_j\left(\left(\sum_i\left\|dx^i\varepsilon_j^{-1}\right\|_{C^r(U_j)}+mp^mC_{j,r,\ell}\right)\times\right.\\
        &\qquad\qquad\qquad\qquad\qquad\qquad\left.\times\left\|u\left(s_j^{\otimes m}\otimes\varepsilon_j^{\otimes \ell}\right)^{-1}\right\|_{C^{r+1}(U_j)}\right).
    \end{aligned}
    $}\]
    So, for $$c(r+1,\ell)=1+c_{n,r}c(r,\ell+1)\max_j\left(\sum_i\left\|dx^i\varepsilon_j^{-1}\right\|_{C^r(U_j)}+C_{j,r,\ell}\right),$$
    we have
    \[
    \|u\|_{C^{r+1}_{\text{int}}}=\|u\|_{C^0}+\left\|\nabla^{m,\ell}u\right\|_{C^r_{\text{int}}}\le c(r+1,\ell)(mp^m)^{r+1}\|u\|_{C^{r+1}_{\text{ext}}}.
    \]
    
    For the lower bound, we will work similarly by induction. We have an analog of \eqref{eq:thetajmell} for $\|\theta_{j,m,\ell}dx^{-1}\|_{C^r}$.
    Call the corresponding constant $C_{j,r,\ell}'$. 
    
    Let $\left\{f_i:U\to \R^{m_i}\right\}_{i=1}^q$ be functions on an open set in Euclidean space, and let $f=f_1\oplus\cdots\oplus f_q:U\to\R^{m_1+\cdots+m_q}$. Then we have
    \begin{equation}\label{eq:ell1ell2stuffandthings}
    \begin{aligned}
    \sum_i\|f_i\|_{C^r}&=\sum_{|\alpha|\le r}\sum_i\sup\left|\frac{\partial^\alpha}{\partial x^\alpha}f_i\right|\le \sum _{|\alpha|\le r}q\sup\sum _i\left|\frac{\partial^\alpha}{\partial x^\alpha}f_i\right|\\
    &\le q\sum _{|\alpha|\le r}q^{1/2}\sup\left(\sum_i\left|\frac{\partial^\alpha}{\partial x^\alpha}f_i\right|^2\right)^{1/2}=q^{3/2}\sup\sum_{|\alpha|\le r}\left|\frac{\partial^\alpha}{\partial x^\alpha}f\right|\\
    &=q^{3/2}\|f\|_{C^r}.
    \end{aligned}
    \end{equation}
    Now, for a section $u$ of $E^{\otimes m}\otimes (T^*M)^{\otimes\ell}$ and for each $j$ we calculate
    \begin{equation*}
        \begin{aligned}
            \sum_i&\left\|\frac{\partial}{\partial x^i}\left(u\left(s_j^{\otimes m}\otimes \varepsilon_j^{\otimes \ell}\right)^{-1}\right) \right\|_{C^r(U_j)}\le\\
            &\overset{\mathclap{\eqref{eq:ell1ell2stuffandthings}}}{\le} n^{3/2} \left\|d\left(u\left(s_j^{\otimes m}\otimes \varepsilon_j^{\otimes \ell}\right)^{-1}\right)dx^{-1}\right\|_{C^r(U_j)}\\
            &\le n^{3/2}\left\|\left(\nabla^{m,\ell}u\cdot\left(s_j^{\otimes m}\otimes \varepsilon_j^{\otimes \ell}\right)^{-1}-\right.\right.\\
            &\qquad\qquad\ \left.\left.-\ \!u\left(s_j^{\otimes m}\otimes \varepsilon_j^{\otimes \ell}\right)^{-1}\theta_{j,m,\ell}\right)dx^{-1}\right\|_{C^r(U_j)}\\
            &\le n^{3/2}\left\|\nabla^{m,\ell}u\cdot\left(s_j^{\otimes m}\otimes \varepsilon_j^{\otimes (\ell+1)}\right)^{-1}\varepsilon_jdx^{-1}\right\|_{C^r(U_j)}\\
            &\qquad+n^{3/2}\left\|u\cdot\left(s_j^{\otimes m}\otimes \varepsilon_j^{\otimes \ell}\right)^{-1}\theta_{j,m,\ell}dx^{-1}\right\|_{C^r(U_j)}\\
            &\le n^{3/2}c_{n,r}\left(\left\|\nabla^{m,\ell}u\right\|_{C^r_{\text{ext}}}\left\|\varepsilon_jdx^{-1}\right\|_{C^r(U_j)}+\|u\|_{C^r_{\text{ext}}}mp^mC'_{j,r,\ell}\right)\\
            &\le n^{3/2}c_{n,r}\left(\left\|\varepsilon_jdx^{-1}\right\|_{C^r(U_j)}c'(r,\ell+1)(mp^m)^r\right.\\
            &\qquad\qquad\quad\ +\ \!C_{j,r,\ell}'c'(r,\ell)(mp^m)^{r+1}\Big)\|u\|_{C^{r+1}_{\text{int}}}.
        \end{aligned}
    \end{equation*}
    So, for
    \[\resizebox{\linewidth}{!}{
    $c'(r+1,\ell):=1+n^{3/2}c_{n,r}\max_j\left(\left\|\varepsilon_jdx^{-1}\right\|_{C^r(U_j)}c'(r,\ell+1)+C_{j,r,\ell}'c'(r,\ell)\right),
    $}
    \]
    we have
    \begin{align*}
    \|u\|_{C^{r+1}_{\text{ext}}}&\le \|u\|_{C^0}+\max_j\sum_i\left\|\frac{\partial}{\partial x^i}\left(u\left(s_j^{\otimes m}\otimes \varepsilon_j^{\otimes \ell}\right)^{-1}\right) \right\|_{C^r(U_j)}\\
    &\le c'(r+1,\ell)(mp^m)^{r+1}\|u\|_{C^{r+1}_{\text{int}}}
    \end{align*}
\end{proof}

The next Lemma~establishes the asymptotic growth of the constant appearing in the fundamental inequality for the operator $\overline\partial$ on $L^m$ as $m\to\infty$.

For any open set $W \subset \C^n$, we denote by $\|\cdot\|_{l,W}$ the $l$-th Sobolev $L^2$ norm on $W$ with respect to the Euclidean metric and by $H^l(W)$ the associated Sobolev space.
For a nonvanishing section $s$ of a holomorphic line bundle $L$, and $s^m$ the associated section of $L^m,$ we have
\[
\overline{\partial}s^m=ms^{m-1}\overline\partial s=\left( ms^{-1}\overline\partial s\right)s^m.
\]
\begin{lm}\label{lm:fundamental}
Let $(L,h)$ be a holomorphic Hermitian line bundle over an open set $\Omega\subseteq \C^n$. Let $s$ be a unitary section over a subset $O\subseteq \Omega$, let $V \subseteq O$ be relatively compact, and let $\tilde V\subseteq V$ be relatively compact.
Denote by $P_{m} : H^k(O) \to H^{k-1}(O)$ the differential operator
\[
P_{m}F=s^{-m}\overline{\partial}\left(Fs^{m}\right)=\overline{\partial}F+ms^{-1}\overline{\partial}s\cdot F.
\]
For each $k$ there is a constant $c=c\left(k,V,\tilde V\right)$ such that for any function~$F$,
\[
\left\Vert F\right\Vert _{k,\tilde V}\le cm^{k}\left(\left\Vert P_{m}F\right\Vert _{k-1,V}+\left\Vert F\right\Vert _{0,V}\right).
\]
\end{lm}
\begin{proof}
In what follows let $c$ be changing.

Take a sequence $\tilde V=V_k\subseteq \overline V_k\subseteq V_{k-1}\subseteq\cdots \subseteq \overline V_1\subseteq V_0 =V $. We will show by induction that for each $l$,
\[
\|F\|_{l,V_l}\le cm^l\left(\|P_mF\|_{l-1,V}+\|F\|_{0,V} \right)
\]
and the result will follow by substituting $l=k$:

By the fundamental inequality for the elliptic operator $\overline{\partial}$,
we have
\begin{align*}
\left\Vert F\right\Vert_{l,V_l} &\le c\left(\left\Vert \overline{\partial}F\right\Vert _{l-1,V_{l-1}}+\left\Vert F\right\Vert _{l-1,V_{l-1}}\right)\\
&=c\left(\left\Vert P_{m}F-ms^{-1}\overline{\partial}s\cdot F\right\Vert _{l-1,V_{l-1}}+\left\Vert F\right\Vert _{l-1,V_{l-1}}\right)\\
 & \le c\left(\left\Vert P_{m}F\right\Vert _{l-1,V_{l-1}}+\left(1+m\left\Vert s^{-1}\overline{\partial}s\right\Vert _{C^{l-1}(V_{l-1})}\right)\left\Vert F\right\Vert _{l-1,V_{l-1}}\right)\\
 & \le c\left(1+\left\Vert s^{-1}\overline{\partial}s\right\Vert _{C^{l-1}(V_{l-1})}\right)m\left(\left\Vert P_{m}F\right\Vert _{l-1,V}+\left\Vert F\right\Vert _{l-1,V_{l-1}}\right).
\end{align*}

For $l=1$ the last inequality proves the induction base, and for $l\ge 2$ using the induction hypothesis for the term $\left\Vert F\right\Vert _{l-1,V_{l-1}}$ finishes the proof.
\end{proof}

\subsection{An almost-holomorphic approximation}\label{subsec:approx}
We start preparing the proof of Theorem~\ref{tm:approximation2} and use the notation thereof below.

By Lemma~\ref{lm:BSextpara}, let $\tilde e$ be an extension of $e$ to a section of $L$ over $X$
such that $\overline{\partial}\tilde e$ vanishes to $(N-1)$-st order along $\Gamma$,
where $N\in\mathbb{N}$ is to be determined later. Define $\tilde f_1=\tilde e^{-1}\overline{\partial}\tilde e$ where $\tilde e\ne 0$,
so that $\tilde f_1$ vanishes to $(N-1)$-st order along $\Gamma$. Note
that
\[
\overline{\partial}\tilde e^{m}=m\tilde e^{m-1}\overline{\partial}\tilde e=m\tilde f_1\tilde e^{m},
\]
and denote $a=\left|\tilde e\right|$. Use Lemma~\ref{lm:FunExt} to extend $ b$ to a function $\tilde b:X\to\C$ with
$$f_0:=\overline\partial \tilde b$$
vanishing to $(N-1)$-st order along $\Gamma$, and $\supp \tilde b\subseteq V$. Denote 
$$f_1= \tilde b\tilde f_1,$$
which vanishes to $(N-1)$-st order along $\Gamma$.

Take the function $\rho$ from Lemma~\ref{lm:rho} with $S=\Gamma$, defined in a neighborhood $U$ of $\Gamma$ which satisfies Lemmas~\ref{lm:Zrclosed} and~\ref{lm:a}, and such that $\tilde e\ne 0$ in $\overline U$ so $\tilde f_1$ is defined and bounded. Recall from Definition~\ref{df:Zr} that $Z_r=\{z\in U:\rho(z)<r\}$, and recall also the notation $R_m=m^{-1/2}\log m$.

Let $\eta:\mathbb{R}_{\ge0}\to[0,1]$ be a smooth bump function supported
in $[0,1)$ which equals $1$ on $[0,1/2]$. 
Define
\begin{equation}\label{eq:v}
    v=\eta\left(\frac{\rho^{2}}{4R_{m}^{2}}\right) \tilde b\tilde e^{m}
\end{equation}
and
\begin{equation}\label{eq:w}
    w=\overline{\partial}v=\left(\frac{1}{4}\eta'\left(\frac{\rho^{2}}{4R_{m}^{2}}\right)R_{m}^{-2}\tilde b\overline{\partial}\rho^{2}
+\eta\left(\frac{\rho^{2}}{4R_{m}^{2}}\right)\left(f_0+mf_1\right)\right)\tilde e^{m}.
\end{equation}
Note that $v$ is supported in $U$ for large $m$, since by Lemma~\ref{lm:Zrclosed} we have $\overline Z_{2R_m}\subset U$. Hence it is globally defined upon extending by $0$ outside of $U$.
\begin{lm}\label{lm:wbound}
    \[
    \|w\|^2=O\left(m^{2-N}\log^{2N}m\right).
    \]
\end{lm}
\begin{proof}
Using \eqref{eq:w} we have
\begin{align*}
    \left|w\right|^{2}&\le\frac{1}{8}\eta'\left(\frac{\rho^{2}}{4R_{m}^{2}}\right)^{2}R_{m}^{-4}\big|\tilde b\big|^2\left|\overline{\partial}\rho^{2}\right|^{2}a^{2m}
+2\eta\left(\frac{\rho^{2}}{4R_{m}^{2}}\right)^{2}\left|f_0+mf_1\right|^{2}a^{2m}\\
&=:W_1+W_2
\end{align*}
This expression vanishes out of $Z_{2R_{m}}$. The expression for $W_1$ vanishes inside $Z_{R_{m}}$ as $\eta'\left(\rho^{2}/4R_m^2\right)\equiv0$
there. Also, $\left|\overline{\partial}\rho^{2}\right|=O\left(1\right)$ as $\rho^2$ is smooth, and similarly $\big|\tilde b\big|,\eta'=O(1)$. Combining this with Lemma~\ref{lm:a}, we get
\[
W_1=O\left(R_{m}^{-4}a^{2m}\right)=O\left(m^{2}\log^{-4}m\cdot a^{2m} \right)=O\left(m^{2-2\gamma\log m}\log^{-4}m\right).
\]

We can bound
\[
W_2\le 2m^2\eta\left(\frac{\rho^{2}}{4R_{m}^{2}}\right)^{2}\left(\left|f_0\right|+\left|f_1\right|\right)^2a^{2m}.
\]
Now, $\left|f_0\right|+\left|f_1\right|=O\left(\rho^{N}\right)$ because $f_0,f_1$ vanish to order $N-1$ along~$\Gamma$, and $a\le1$ in $Z_{2R_{m}}$, so
\[
W_2=O\left(m^{2}R_{m}^{2N}\right)=O\left(m^{2-N}\log^{2N}m\right).
\]

For large enough $m$ the estimate on $W_2$ dominates the one on $W_1$. The
volume of $Z_{2R_{m}}$ is $O(1)$ (in fact, $O(R_m^n)$, but we don't need it),
so
\begin{align*}
\left\Vert w\right\Vert ^{2} & =\int_{X}\left|w\right|^{2}\vol_g=\int_{Z_{2R_{m}}}\left|w\right|^{2}\vol_g
  \le  O\left(m^{2-N}\log^{2N}m\right).
\end{align*}
\end{proof}
Cover $X$ by a finite family of holomorphic coordinate charts $\left\{ O_{t}\right\} $
on each of which ${L}$ is trivial. For some $\ell$ we may assume that $U=\bigcup_{t=1}^{\ell}O_{t}$
and that $O_{t}\cap\supp v=\emptyset$ for $t>\ell$ (and large enough
$m$). Let $\left\{ V_{t}\right\} $ be a cover of $X$ with $V_{t}$
relatively compact in $O_{t}$ for every $t$.
For each $t$ let $s_{t}$ be a unitary frame over $O_{t}$, and assume
that $s_{t}=\tilde e/a$ for $t\le\ell$. Define $w_{t}=s_{t}^{-m}w$.
We identify each $O_t$ and $V_t$ with an open set in $\C^n$.
\begin{lm}\label{lm:wderbound}
    Let $k<N$. Then for every multi-index $\alpha$ with $|\alpha|\le k$, the $\alpha$-th derivative of $w_t$ satisfies the $L^2$-bound
    \[
    \left\|\frac{\partial^\alpha}{\partial x^\alpha} w_t\right\|_{L^2(V_t)}=O\left(m^{1-N/2+2k}\log^{N-2k}m\right).
    \]
As a consequence,
    \[
    \|w_t\|_{k,V_t}=O\left(m^{1-N/2+2k}\log^{N-2k}m\right).
    \]
\end{lm}
\begin{proof}
   The assertion is trivial for $t>\ell$ since $w_t=0$. Now assume $t\le\ell$.
    Recall that $s_{t}=\tilde e/a$, so that by~\eqref{eq:w},
\begin{align*}
w_{t}&=\frac{1}{4}\eta'\left(\frac{\rho^{2}}{4R_{m}^{2}}\right)R_{m}^{-2}\tilde b\overline{\partial}\rho^{2}a^{m}+\eta\left(\frac{\rho^{2}}{4R_{m}^{2}}\right)f_0a^{m}+\eta\left(\frac{\rho^{2}}{4R_{m}^{2}}\right)mf_1a^{m}\\
&=:w_{t1}+w_{t2}+w_{t3}.
\end{align*}

The term $\frac{\partial^\alpha}{\partial x^\alpha}w_{t1}$ is a combination of expressions, each of which
contains derivatives of the factors $a^{m},\eta'\left(\rho^{2}/4R_{m}^{2}\right)$ and $\tilde b$ of order no more than $k$.
Therefore it is a product $\delta_{1}\frac{m!}{(m-k)!}a^{m-k}$,
where $\delta_{1}$ is a function in $O\left(R_{m}^{-2(1+k)}\right)$.
Since $k$
is independent of $m$, we have $a^{-k}=O(1)$. Since $w_{t1}\equiv0$ in $Z_{R_{m}}$, we can apply Lemma~\ref{lm:a} and get
\[
\left|\frac{\partial^\alpha}{\partial x^\alpha}w_{t1}\right|=O\left(R_{m}^{-2(1+k)}\right)\cdot O\left(m^{k}\right)\cdot a^{m}=O\left(R_{m}^{-2(1+k)}m^{k-\gamma\log m}\right).
\]

The treatment of $\frac{\partial^\alpha}{\partial x^\alpha}w_{t2}$ is similar: given multi-indices $\alpha=(\alpha_1,...,\alpha_{2n}),\beta=(\beta_1,...,\beta_{2n})$, we write $\beta\le \alpha$ to indicate $\beta_j\le \alpha_j$ for each $1\le j\le 2n$. 
Then
\[
\frac{\partial^\alpha}{\partial x^\alpha}w_{t2}=\sum_{\beta\le\alpha}\delta_{\alpha-\beta}\frac{m!}{(m-|\alpha|-|\beta|)!}a^{m-|\alpha|-|\beta|}\frac{\partial^\beta}{\partial x^\beta}f_0,
\]
where each $\delta_{\kappa}$ is $O\left(R_{m}^{-2|\kappa|}\right)$.
Since $N>k$, $\frac{\partial^\beta}{\partial x^\beta}f_0$
vanishes on $\Gamma$ to $(N-\left|\beta\right|+1)$-st order, hence
is $O\left(R_{m}^{N-|\beta|}\right)$, so
\begin{align*}
\left|\delta_{\alpha-\beta}\frac{m!}{(m-|\alpha-\beta|)!}a^{m-|\alpha-\beta|}\frac{\partial^\beta}{\partial x^\beta}f_0\right| & =O\left(R_{m}^{-2|\alpha-\beta|}m^{|\alpha-\beta|}R_{m}^{N-|\beta|}\right)\\
&=O\left(m^{|\alpha|-|\beta|}R_{m}^{N-2|\alpha|+|\beta|}\right)\\
 & =O\left(m^kR_{m}^{N-2k}\right),
\end{align*}
which gives
\[
\left|\frac{\partial^\alpha}{\partial x^\alpha}w_{t2}\right|=O\left(m^kR_{m}^{N-2k}\right).
\]
The analysis of $\frac{\partial^\alpha}{\partial x^\alpha} w_{t3}$ is identical, but there is an additional $m$ factor, so
\[
\left|\frac{\partial^\alpha}{\partial x^\alpha}w_{t3}\right|=O\left(m^{k+1}R_{m}^{N-2k}\right).
\]
For $m$ large enough this estimate dominates the other two,
so altogether we have
\begin{align*}
    \left|\frac{\partial^\alpha}{\partial x^\alpha}w_{t}\right|^{2}=O\left(m^{2+2k}R_{m}^{2N-4k}\right),
\end{align*}
and
\[
\left\Vert \frac{\partial^\alpha}{\partial x^\alpha}w_{t}\right\Vert _{L^2\left(V_t\right)}^{2}=O\left(m^{2+2k}R_{m}^{2N-4k}\right)=O\left(m^{2-N+4k}\log^{2N-4k}m\right).
\]
\end{proof}
\subsection{Final construction and proof}
We continue with the notation of Sections~\ref{subsec:lemmas} and~\ref{subsec:approx}.
Let $C>0$ be a constant such that $\i \ric g\ge-C\omega$. Then,
\begin{align*}
    \i\ric g+\i F_{h^{m}}&=\i\ric g+2\pi m\omega\\
    &\ge 2\pi m\left(1-\frac{C}{2\pi m}\right)\omega\ge\pi m\omega
\end{align*}
for large $m$.

By Lemmas~\ref{lm:3.1} and~\ref{lm:wbound}, there is a section $u$ of $L^m$
satisfying $\overline{\partial}u=w$ and
\begin{equation}\label{eq:ineq}
    \left\Vert u\right\Vert ^{2}\le\frac{1}{\pi m}\left\Vert w\right\Vert ^{2}=O\left(m^{1-N}\log^{2N}m\right).
\end{equation}
Let $\left\{\tilde V_t\right\}$ be a cover with $\overline{\tilde V_t}\subseteq V_t$ for each $t$, and define $u_t=s_t^{-m}u$ for each $t$.

\begin{lm}\label{lm:ubound}
For $N>n+r$ we have 
\begin{align*}
\|u\|_{C^r_\mathrm{int}}&=O\left(m^{r+3(n+r)+2-N/2}\log^{N-2(n+r)}m\right),\\
\|u|_\Gamma\|_{C^r_\mathrm{int}}&=O\left(m^{3r+3(n+r)+2-N/2}\log^{N-2(n+r)}m\right).
\end{align*}
In particular, for $N>2r+6(n+r)+4$ we have $\|u\|_{C^r_\mathrm{int}}=o(1)$, and for $N>6r+6(n+r)+4$ we have $\|u|_\Gamma\|_{C^r_\mathrm{int}}=o(1)$.

\end{lm}
\begin{proof}
Using Lemma~\ref{lm:Cr} for the line bundle $L$ so $p=1$, we get $\|u\|_{C^r_\mathrm{int}}=O(m^r)\cdot\max_t\|u_t\|_{C^r\left(\tilde V_t\right)}$, so for the estimate on $\|u\|_{C^r_\mathrm{int}}$ it suffices to show that $\|u_t\|_{C^r\left(\tilde V_t\right)}=O\left(m^{2(n+r)+2-N/2}\log^{N-2n}m\right)$ for each~$t$.

Let $P_{t,m}$ be the operator $P_{m}$ from Lemma~\ref{lm:fundamental} with respect to the frame $s_t$, i.e. $P_{t,m}F=s_t^{-1}\overline\partial(Fs_t) $. Note that $P_{t,m}u_{t}=w_{t}$.
By Lemma~\ref{lm:fundamental} for $k=n+r+1$, there is a constant $c_t$ such that
\begin{equation}\label{ineq:utwt}
    \|u_t\|_{n+r+1,\tilde V_t}\le c_tm^{n+r+1}\left(\|w_t\|_{n+r,V_t}+\|u_t\|_{0,V_t} \right).
\end{equation}
Let $G_t>0$ be an upper bound on $\left(\det(g_{ij})\right)^{-1}$ in $V_t$, and let $\vol_{\text{st}}=\left(\frac{i}{2}\right)^n\bigwedge\left(dz_{j}\wedge d\overline{z}_{j}\right)$
be the standard volume form on $\mathbb{C}^{n}$. Then
\[
\|u_t\|_{0,V_t}=\left(\int_{V_t}|u_t|^2\vol_\text{st}\right)^{1/2}\le G_t^{1/2}\left(\int_{V_t}|u|_{h}^2\vol_g\right)^{1/2}\le G_t^{1/2}\|u\|.
\]
Combining this with inequalities~\eqref{eq:ineq} and~\eqref{ineq:utwt} and Lemma~\ref{lm:wderbound} with $k=n+r$, we get
\begin{align*}
\left\Vert u_{t}\right\Vert _{n+r+1,\tilde V_t}&=c_tm^{n+r+1}\left(O\left(m^{1-N/2+2(n+r)}\log^{N-2(n+r)}m\right)\right.\\
&\qquad\qquad\qquad+O\left(m^{1/2-N/2}\log^{N}m\right)\!\Big).
\end{align*}
The first summand is dominant, so
\begin{align*}
    \left\Vert u_{t}\right\Vert _{n+r+1,\tilde V_t}&=m^{n+r+1}\cdot O\left(m^{1-N/2+2(n+r)}\log^{N-2(n+r)}m\right)\\
    &=O\left(m^{3(n+r)+2-N/2}\log^{N-2(n+r)}m\right).
\end{align*}

Now since $2\cdot (n+r+1)>2n+2\cdot r$, we can apply the Sobolev embedding theorem on the set $\tilde V_t$, and taking the Sobolev constant $K$ to be the maximum over $t$, we get that
\begin{align*}
    \|u_t\|_{C^r\left(\tilde V_t\right)}\le K\|u_t\|_{n+r+1,\tilde V_t}=O\left(m^{3(n+r)+2-N/2}\log^{N-2(n+r)}m\right),
\end{align*}
and this finishes the estimate on $\|u\|_{C^r_\mathrm{int}}$.

Now cover $\Gamma$ by slice charts $U_j$ relatively compact in larger slice charts, on which $L$ is trivial. We identify each $U_j$ with an open set in $\C^n$, and $U_j\cap \Gamma$ with $U_j\cap \R^n$. Taking again orthonormal frames $s_j$ and defining $u_j=s_j^{-1}u$, we have
\begin{equation*}
    \begin{aligned}
    \|u|_\Gamma\|_{C^r_\mathrm{int}}&\overset{{\text{Lem.}~\ref{lm:Cr}}}{=}O(m^r)\cdot\max_j\|u_j|_{U_j\cap \R^n}\|_{C^r}\\
    &\overset{\hphantom{\text{Lem.}~\ref{lm:Cr}}}{=}O(m^r)\cdot\max_j\sum_{\substack{\alpha=(\alpha_1,...,\alpha_n)\\|\alpha|\le r}}\sup_{U_j\cap \R^n}\left|\frac{\partial^\alpha}{\partial x^\alpha}u_j\right|\\
    &\overset{\hphantom{\text{Lem.}~\ref{lm:Cr}}}{\le} O(m^r)\cdot\max_j\sum_{\substack{\alpha=(\alpha_1,...,\alpha_{2n})\\|\alpha|\le r}}\sup_{U_j}\left|\frac{\partial^\alpha}{\partial x^\alpha}u_j\right|\\
    &\overset{{\text{Lem.}~\ref{lm:Cr}}}{=} O(m^{2r})\|u\|_{C^r_\mathrm{int}}\\
    &\overset{\hphantom{\text{Lem.}~\ref{lm:Cr}}}{=}O\left(m^{3r+3(n+r)+2-N/2}\log^{N-2(n+r)}m\right).
\end{aligned}
\end{equation*}
\end{proof}
\begin{proof}[Proof of Theorem~\ref{tm:approximation2}]
    Recall the definition of $v$ in equation \eqref{eq:v}.
    Take a large enough $N$ such that the second assertion of Lemma~\ref{lm:ubound} holds.
    Let $s=v-u$, so that $\overline{\partial}s=w-w=0$, i.e. $s$ is holomorphic. Then $s|_{\Gamma}-be^m$ is just $-u|_{\Gamma}$, and for large enough $m$ Lemma~\ref{lm:ubound} implies that its $C^r$-norm is less than $\varepsilon$. Moreover, outside of $V$ we have $s=-u$, so again its $C^r$-norm there is less than $\varepsilon$. 
\end{proof}

\section{Proofs of the main theorems}\label{sec:Pfs}
\subsection{Quantization separates positive Lagrangians}
\begin{proof}[Proof of Theorem~\ref{tm:separation}]
Fix some $e_j\in \cP_{\Gamma_j}$ for $j=0,1$.
Let $\emptyset\ne U\subset\overline U\subseteq \Gamma_0\setminus \Gamma_1$ be an open set, and let $b:\Gamma_0\to [0,1]$ be a bump function supported in $\Gamma_0\setminus\Gamma_1$ with $b|_{U}\equiv 1$.

For any $\varepsilon>0$ take $s\in \H_m$ from Theorem~\ref{tm:approximation1} with $V=X\setminus\Gamma_1$ and $r=0$, for large enough $m\in\N$. Then $|s|_h<\varepsilon$ on $\Gamma_1$ and $|s-be_0^m|_h<\varepsilon$ on $\Gamma_0$.

Let $M$ be a bound for $\left|\Omega\right|_g$. Then
\begin{align*}
    \left|\psi_{\Gamma_1,e_1}(s)\right|=\left|\int_{\Gamma_1}\frac{s}{e_1^m}\Omega \right|\le\int_{\Gamma_1}|s|_h|\Omega|_g\vol_{\Gamma_1}< M\vol(\Gamma_1)\varepsilon
\end{align*}
and
\begin{align*}
    \left|\psi_{\Gamma_0,e_0}(s)\right|&=\left|\int_{\Gamma_0}\frac{be_0^m}{e_0^m}\Omega+\int_{\Gamma_0}\frac{s-be_0^m}{e_0^m}\Omega \right|\\
    &\ge \left|\int_{\Gamma_0} b\Omega\right|-\int_{\Gamma_0}\left| s-be_0^m\right|_hM\vol_{\Gamma_0}\\
    &> \re\int_{\Gamma_0} b\Omega-M\vol(\Gamma_0)\varepsilon\\
    &\ge\int_{U}\re\Omega-M\vol(\Gamma_0)\varepsilon.
\end{align*}
For small enough $\varepsilon$ we have $\left|\psi_{\Gamma_1,e_1}(s)\right|<\left|\psi_{\Gamma_0,e_0}(s)\right|$. Hence, for all $z,w\in S^1$ we have
\[
\varphi_m(\Gamma_0,ze_0)(s)=z^{-m}\psi_{\Gamma_0,e_0}(s) \ne w^{-m}\psi_{\Gamma_1,e_1}(s)=\varphi_m(\Gamma_1,we_1)(s).
\]

\end{proof}

\subsection{Non-degeneracy of the metric}
\begin{proof}[Proof of Theorem~\ref{tm:pm}]
    Let $\Gamma_0,\Gamma_1\in \OO^\epsilon$. By Theorem~\ref{tm:separation}, for $m$ sufficiently large, $\varphi_m(\cP_{\Gamma_0})$ and $\varphi_m(\cP_{\Gamma_1})$ are disjoint compact subsets of~$\H^\vee_m$. Let $\delta>0$ be the distance between them.

    Let $\gamma(t)=\Gamma_t,t\in [0,1]$ be any piecewise smooth path in $\OO^\epsilon$ between $\Gamma_0,\Gamma_1$. Recall the flat structure on $\cP$ from Definition~\ref{df:cP}, the horizontal subspaces of which are given explicitly by Lemma~\ref{lm:connection}. Take a horizontal lift $\tilde \gamma(t)=(\Gamma_t,e_t)$ in $\cP$.
    By Corollary~\ref{cy:derivatives}, under the canonical identifications $T_{(\Gamma,e)}\cP=C^\infty(\Gamma) $ and $T_\Gamma\OO=\T{\Gamma}=\{h:\int h\re\Omega=0\},$ we have
    
    $$\dot\gamma(t)=d\pi_{\tilde\gamma(t)}\dot{\tilde\gamma}(t)=\dot{\tilde\gamma}(t).$$ 
    Hence,
    \[
    \ell(\gamma)=\int \sqrt{G\left(\dot\gamma(t),\dot\gamma(t)\right)}dt=\int\sqrt{\overline G\left(\dot{\tilde\gamma}(t),\dot{\tilde\gamma}(t)\right)}dt=\ell\left(\tilde\gamma\right).
    \]
Thus, by Theorem~\ref{tm:bound} we have
    \begin{align*}
        \delta&\le \ell\left(\varphi_m\circ\tilde\gamma\right)=\int\left\|d(\varphi_m)_{\tilde{\gamma}(t)}\dot{\tilde\gamma}(t)\right\|dt\\
        &\le C_m\int\left\|\dot{\tilde\gamma}(t)\right\|dt=C_m\ell(\tilde\gamma)\\
        &=C_m\ell(\gamma). 
    \end{align*}
That is, $\ell(\gamma)\ge \frac{\delta}{C_m}$. Therefore,
    \[
    d(\Gamma_0,\Gamma_1)=\inf_{\gamma}\ell(\gamma)\ge\frac{\delta}{C_m}>0.
    \]
\end{proof}

\bibliographystyle{amsabbrvcnobysame}
\bibliography{ref}

\end{document}